\documentclass[12pt]{extarticle}
\usepackage{amsmath, amsthm, amssymb, color}
\usepackage[colorlinks=true,linkcolor=blue,urlcolor=blue]{hyperref}
\usepackage{graphicx}
\usepackage{caption}
\usepackage{mathtools}
\usepackage{enumerate}
\usepackage{verbatim}
\usepackage{tikz,tikz-cd,tikz-3dplot}
\usepackage{amssymb}
\usetikzlibrary{matrix}
\usetikzlibrary{arrows}
\usepackage{algorithm}
\usepackage[noend]{algpseudocode}
\usepackage{caption}
\usepackage[normalem]{ulem}
\usepackage{subcaption}
\oddsidemargin \evensidemargin
\usepackage{makecell}
\usepackage{array}
\usepackage{enumitem}
\newtheorem{theorem}{Theorem}[section]
\newtheorem{proposition}[theorem]{Proposition}
\newtheorem{lemma}[theorem]{Lemma}
\newtheorem{corollary}[theorem]{Corollary}

\theoremstyle{definition}
\newtheorem{definition}[theorem]{Definition}

\newtheorem{remark}[theorem]{Remark}

\newtheorem{examplex}[theorem]{Example}
\newenvironment{example}
{\pushQED{\qed}\examplex}
{\popQED\endexamplex}

\makeatletter
\def\Ddots{\mathinner{\mkern1mu\raise\p@
\vbox{\kern7\p@\hbox{.}}\mkern2mu
\raise4\p@\hbox{.}\mkern2mu\raise7\p@\hbox{.}\mkern1mu}}
\makeatother

\newcommand{\PP}{\mathbb{P}}
\newcommand{\RR}{\mathbb{R}}
\newcommand{\QQ}{\mathbb{Q}}
\newcommand{\CC}{\mathbb{C} }
\newcommand{\ZZ}{\mathbb{Z}}

\title{\bf Gibbs Manifolds}

\author{Dmitrii Pavlov,
  Bernd Sturmfels and Simon Telen}
\date{}
\begin{document}
\maketitle

\begin{abstract}
\noindent
Gibbs manifolds are images of affine
spaces of symmetric matrices  under the 
exponential map. They arise in applications such as optimization, 
statistics and quantum~physics, where they
extend the ubiquitous role of toric geometry.
The Gibbs variety is the zero locus of all polynomials
that vanish on the Gibbs manifold.
We compute these polynomials and show that the
 Gibbs variety is low-dimensional.
Our theory is applied to a wide range of
scenarios, including matrix pencils and
quantum optimal transport.
    
\end{abstract}

\section{Introduction} \label{sec:1}

Toric varieties provide the geometric foundations for 
many successes in the mathematical sciences.
In statistics they appear as discrete exponential families \cite[p.~2]{Seth}, and
their ideals reveal Markov bases
for sampling from conditional distributions \cite{DS}.
In optimization, they furnish
nonnegativity certificates \cite{FDW}
and they govern the
entropic regularization of linear programming \cite{ourteam}.
Notable sightings in phylogenetics, stochastic analysis, 
Gaussian inference and chemical reaction networks
gave us the slogan that
{\em the world is toric} \cite[Section~8.3]{MS}.

In all of these applications, the key player is the
positive part of the toric variety. That real manifold
is identified with a convex polytope
by the moment map \cite[Theorem~8.24]{MS}.
The fibers of the underlying linear map are polytopes
of complementary dimension, and each fiber
intersects the toric variety uniquely, in the
{\em Birch point}.
This is the unique maximizer of the
entropy over the polytope \cite[Theorem 1.10]{ASCB}.
 In statistical physics and computer science \cite{Vigo},
the Birch point is  known as the {\em Gibbs distribution}.
The name Gibbs  refers to the maximum entropy state
in a quantum system, and this is also the reason behind our~title.

This paper initiates a non-commutative extension of applied toric geometry.
In that extension, points in $\RR^n$ are replaced by real symmetric
$n \times n$ matrices, and
linear programing is replaced by semidefinite programming.
There is a moment map which takes the cone of positive semidefinite matrices
onto a spectrahedral shadow, and whose fibers are spectrahedra
of complementary dimension.  The
Gibbs manifold plays the role of the positive toric variety.
Each spectrahedron intersects the Gibbs manifold uniquely,
in the {\em Gibbs point}. Just like in the toric case, we
study these objects algebraically by passing to the
Zariski closure of our positive manifold.
The resulting analogues of toric varieties are called
 {\em Gibbs~varieties}.

We illustrate these concepts for
   the following linear space of
 symmetric $3 \times 3$-matrices:
\begin{equation} \label{eq:Lintro}
\mathcal{L} \,\,\, =\,\,\, \biggl\{\,
\begin{small}
\begin{bmatrix} y_1+y_2+y_3  & y_1 & y_2 \\
                             y_1   & y_1+y_2+y_3 & y_3 \\
                           y_2 & y_3 & y_1+y_2+y_3
        \end{bmatrix} \end{small} \,\,: \,\, y_1,y_2,y_3 \in \RR\, \biggr\}.
\end{equation}
The Gibbs manifold ${\rm GM}(\mathcal{L})$ is obtained by
applying the exponential function to each matrix in $\mathcal{L}$.
 Since the matrix logarithm is the inverse to the matrix exponential,
it is a $3$-dimensional manifold,
contained in the $6$-dimensional
cone ${\rm int}(\mathbb{S}^3_+)$ of positive definite $3 \times 3$ matrices.

Consider the quotient map  from the matrix space $\mathbb{S}^3 \simeq \RR^6$
onto $\mathbb{S}^3/\mathcal{L} \simeq \RR^3$.
This takes~a positive semidefinite matrix
 $X = [x_{ij}]$ to its inner products with the matrices in a basis of~$\mathcal{L}$:
$$ \pi: \,\mathbb{S}^3_+ \rightarrow \RR^3 \, \,:\,
X \, \mapsto \, \bigl(\,{\rm trace}(X) + 2x_{12},\, {\rm trace}(X) + 2x_{13} , \,{\rm trace}(X) + 2x_{23} \,\bigr). $$
Precisely this map appeared in the statistical study of Gaussian models in
\cite[Example 1.1]{SU}.
The fibers $\pi^{-1}(b)$ are three-dimensional spectrahedra, and these
serve as feasible regions in optimization, both for semidefinite programming and for
maximum likelihood estimation.

We here consider yet another convex
optimization problem over the spectrahedron $\pi^{-1}(b)$, namely 
maximizing the {\em von Neumann entropy}
$\, h(X) = {\rm trace}(X- X \cdot {\rm log}(X))$. 
This~problem has a unique local and global maximum, at the intersection
$\pi^{-1}(b) \,\cap\,{\rm GM(\mathcal{L})}$. See Theorem~\ref{thm:opt}.
This {\em Gibbs point} is the  maximizer of the entropy 
over the spectrahedron.
Therefore, the Gibbs manifold ${\rm GM}(\mathcal{L})$ is the
set of Gibbs points in all fibers $\pi^{-1}(b)$, as $b$ ranges over~$\RR^3$.

To study these objects algebraically, we ask for the polynomials
that vanish on ${\rm GM}(\mathcal{L})$. The zeros of these
polynomials form the {\em Gibbs variety}
 ${\rm GV}(\mathcal{L})$.
Thus, the Gibbs variety  is
the Zariski closure of the Gibbs manifold. In our example, the Gibbs manifold
has dimension~$3$, whereas  the Gibbs variety has dimension $5$.
The latter is the cubic hypersurface
$\,
{\rm GV}(\mathcal{L}) \,= \,$
$$
\bigl\{X \in \mathbb{S}^3 :
(x_{11}-x_{22})(x_{11}-x_{33})(x_{22}-x_{33}) = x_{33}(x_{13}^2-x_{23}^2)
+ x_{22}(x_{23}^2-x_{12}^2) + x_{11}(x_{12}^2 -x_{13}^2)  \bigr\}. $$

As promised, the study of Gibbs manifolds and Gibbs varieties is a 
non-commutative extension of applied toric geometry.
Indeed, every  toric variety is a Gibbs variety arising from  
diagonal matrices.
For instance, the toric surface
$\,\{ \,x \in \RR^3: x_1 x_3 \,=\, x_2^2\, \}\,$ is realized as 
$$
 {\rm GV}(\mathcal{L}') \, = \,
\bigl\{X \in \mathbb{S}^3 \,:\,
x_{11} x_{33} - x_{22}^2 \, = \, x_{12} = x_{13} = x_{23} = 0 \,\bigr\} 
$$
for the diagonal matrix pencil
\begin{equation}
\label{eq:toricsurface}
 \mathcal{L}' \, = \, 
\begin{small} \biggl\{\,
\begin{bmatrix} 2 y_1 & 0 & 0 \\
                             0  & y_1{+}y_2 & 0 \\
                             0 & 0 & 2 y_2 
        \end{bmatrix} \end{small} \,\,: \,\, y_1,y_2 \in \RR\, \biggr\}.
\end{equation}
However, even for diagonal matrices, the dimension of the Gibbs
variety can exceed that of the Gibbs manifold.
To see this, replace the matrix entry
$2 y_1$ by $\sqrt{2} y_1$ in the definition of $\mathcal{L}'$.
This explains why transcendental number theory will make an appearance in our study.

\smallskip

Our presentation in this paper is organized as follows.
Section \ref{sec:2} gives a more thorough introduction to 
 Gibbs manifolds and Gibbs varieties.
Theorem \ref{thm:dimGV} states that the dimension of the Gibbs variety is
usually quite small. The proof of this result is presented in Section~\ref{sec:3}.
In that section we present algorithms for computing the prime ideal of the Gibbs variety.
This is an implicitization problem, where the parametrization uses transcendental functions.
We compare
exact symbolic methods for solving that problem with a numerical approach.
A key ingredient is the Galois group for
the eigenvalues of a linear space of symmetric matrices. We implemented our algorithms in \texttt{Julia}, making use of the computer algebra package \texttt{Oscar.jl} \cite{Oscar}. Our code and data are available at \url{https://mathrepo.mis.mpg.de/GibbsManifolds}.

In Section \ref{sec:4} we study the Gibbs varieties given by two-dimensional
spaces of symmetric matrices. This rests on the classical 
Segre-Kronecker classification of matrix pencils \cite{FMS}.

In Section \ref{sec:5} we turn to the application that led us to start this project,
namely  {\em entropic regularization} in convex optimization. That section develops
the natural generalization of the geometric
results  in \cite{ourteam} from linear programming to semidefinite programming.
We conclude in Section \ref{sec:6} with a study of 
{\em quantum optimal transport}  \cite{CEFZ}. This is the
semidefinite programming analogue to the classical
optimal transport problem \cite[Section~3]{ourteam}.
 We show that its
Gibbs manifold is the positive part of a  Segre variety in matrix~space.

\section{From manifolds to varieties} \label{sec:2}

We write $\mathbb{S}^n$ for the space of symmetric $n \times n$-matrices.
This is a real vector space of dimension $\binom{n+1}{2}$. The subset
of positive semidefinite matrices is denoted $\mathbb{S}^n_+$.
 This is a full-dimensional  closed semialgebraic convex cone  
in $\mathbb{S}^n$, known as the {\em PSD cone}. The PSD cone is self-dual with respect to the trace inner product, given by
 $\, \langle X , Y  \rangle  \,:=\, \mathrm{trace}(X Y)$ for $X, Y \in \mathbb{S}^n$.

The matrix exponential function is defined
by the usual power series, which converges for all  real and complex $n \times n$ matrices.
It maps symmetric matrices to positive definite symmetric matrices. 
The zero matrix $0_n$ is mapped to the identity matrix ${\rm id}_n$. We write
$$ {\rm exp} \,\,:\, \mathbb{S}^n \rightarrow {\rm int} (\mathbb{S}^n_+)\,, \,\,
X \,\mapsto \, \sum_{i=0}^\infty\, \frac{1}{i !} \,X^i . $$
This map is invertible, with the inverse given by the familiar series for the logarithm:
$$ {\rm log}\,\, :\, {\rm int} (\mathbb{S}^n_+) \rightarrow \mathbb{S}^n \,,\,\,
Y \,\mapsto \, \sum_{j=1}^\infty \frac{(-1)^{j-1}}{j} \,( \,Y - {\rm id}_n)^j.
$$

We next introduce the geometric objects studied in this article. We fix any 
matrix $A_0 \in \mathbb{S}^n$ and $d$ linearly independent matrices 
$A_1, A_2,\ldots,A_d$, also in $\mathbb{S}^n$.
 We write $\mathcal{L}$ for the affine subspace $\,A_0 + {\rm span}_{\mathbb{R}}(A_1, \ldots, A_d)\,$ 
 of the vector space $\mathbb{S}^n \simeq \RR^{\binom{n+1}{2}}$. Thus,
$\mathcal{L}$ is an {\em affine space of symmetric matrices} (ASSM) of dimension $d$. If $A_0 = 0$, 
then ${\cal L}$ is a {\em linear space~of symmetric matrices} (LSSM). 
We are interested in the image 
of ${\cal L}$ under the exponential map:

\begin{definition}
The {\em Gibbs manifold} ${\rm GM}({\cal L})$ of $\mathcal{L}$ is the $d$-dimensional manifold ${\rm exp}(\mathcal{L}) \subset \mathbb{S}^n_+$.
\end{definition}
This is indeed a $d$-dimensional manifold inside the convex cone $\mathbb{S}^n_+$. It is diffeomorphic to
 $\mathcal{L} \simeq \RR^d$, with the identification given by
 the exponential map and the logarithm map.

In notable special cases (e.g.~that in Section \ref{sec:6}), the Gibbs manifold is semi-algebraic, namely it
 is the intersection of an algebraic variety with the PSD cone. However, 
     this fails in general, as seen in the Introduction.
    It is still interesting to ask which polynomial relations hold between the entries of any matrix in ${\rm GM}({\cal L})$. This motivates the following definition. 

\begin{definition}
The {\em Gibbs variety} ${\rm GV}({\cal L})$ of $\mathcal{L}$ is the Zariski closure of ${\rm GM}({\cal L})$ in $\mathbb{C}^{\binom{n+1}{2}}$.
\end{definition}

\begin{example}[$n = 4, d = 2$] \label{ex:n4d2}
Consider the 2-dimensional linear space of symmetric matrices 
\[ {\cal L} \,\,=\,\, \left \{ \, \begin{small} \begin{bmatrix}
    0 & 0 & 0 & y_1 \\ 0 & 0 & y_1 & y_2 \\ 0 & y_1 & y_2 & 0
     \\ y_1 & y_2 & 0 & 0
\end{bmatrix} \end{small} \,:\, \,  y_1, y_2 \in \mathbb{R} \,  \right \} \, \,\subset \,\, \mathbb{S}^4.\]
Its Gibbs manifold ${\rm GM}({\cal L})$ is a surface in $\mathbb{S}^4 \simeq \mathbb{R}^{10}$. 
The Gibbs variety ${\rm GV}({\cal L})$ has dimension five and degree three. It consists of all symmetric matrices $X = (x_{ij})$ whose entries satisfy
    \begin{equation} \label{eq:n4d2}
\begin{matrix} \qquad
        x_{13}-x_{22}+x_{44} \,= \,x_{14}-x_{23}+x_{34} \,=\, \,x_{24}-x_{33}+x_{44} \,=\, 0,
        \smallskip \\  {\rm and} \qquad
{\rm rank} 
 \begin{bmatrix}        
   x_{11}{-}x_{44} &  x_{12}{-}x_{34} & x_{22}{-}x_{33} \\
   x_{12} &    x_{22}{-}x_{44} &  x_{23}{-}x_{34}
\end{bmatrix} \,\leq \,1. \qquad
\end{matrix}        
\end{equation}
This follows from the general result on matrix pencils in Theorem \ref{thm:[n]}.
\end{example}

The following dimension bounds constitute our main result on Gibbs varieties.

\begin{theorem} \label{thm:dimGV}
    Let ${\cal L} \subset \mathbb{S}^n$ be an ASSM of dimension $d$. The dimension of the Gibbs variety ${\rm GV}({\cal L})$ is at most $n + d$. If $A_0 = 0$, i.e.~${\cal L}$ is an LSSM, 
    then ${\rm dim} \, {\rm GV}({\cal L})$ is at most $n + d-1$. 
    \end{theorem}

These bounds are attained in many cases, including Example~\ref{ex:n4d2}.
Our proof of Theorem~\ref{thm:dimGV} appears in
 Section \ref{sec:3}, in the context of algorithms for computing the ideal of ${\rm GV}({\cal L})$.

While it might be difficult to find all polynomials 
that vanish on the Gibbs manifold, finding \emph{linear} relations is sometimes easier. Such relations are useful for semidefinite optimization, see Remark \ref{rem:gibbsplaneopt}. This brings us to the final geometric object studied in this~paper.

\begin{definition}
    The {\em Gibbs plane} ${\rm GP}({\cal L})$ is the smallest affine space containing ${\rm GV}({\cal L})$.
\end{definition}

Clearly, we have the chain of inclusions 
$\,{\rm GM}({\cal L}) \,\subseteq \,{\rm GV}({\cal L}) \,\subseteq\, {\rm GP}({\cal L}) \,\subseteq\, \mathbb{C}^{\binom{n+1}{2}}$.

\begin{example}
    The Gibbs plane of the LSSM ${\cal L}$ from Example \ref{ex:n4d2} is the 
    $7$-dimensional linear space in $\CC^{10}$ that is
     defined by the three linear relations listed in the first row of \eqref{eq:n4d2}.
\end{example}

We claimed in the Introduction that this article offers a
generalization of toric varieties. In what follows, we make that
claim precise, by discussing the case when ${\cal L}$ is a commuting family. 
This means that the symmetric matrices $A_0, A_1, \ldots, A_d$ 
 commute pairwise, i.e. $A_i A_j = A_j A_i$ for all $i,j$.
 We now assume that this holds.
 Then the ASSM $\mathcal{L}$  can be diagonalized,
i.e.~there is an orthogonal matrix $V$ such that $\Lambda_i = V^\top A_iV$ is a diagonal matrix, for all $i$. The vector $\lambda_i \in \mathbb{R}^n$ 
of diagonal entries in $\Lambda_i = {\rm diag}(\lambda_i)$ contains the eigenvalues of~$A_i$. 

The matrix exponential of any element in ${\cal L}$ can be computed as follows:
\begin{equation}
\label{eq:changewithV}
 {\rm exp}(A_0 + y_1A_1 + \cdots + y_d A_d) \,\, =
 \, \, V \cdot {\rm exp}(\Lambda_0 + y_1 \Lambda_1 + \cdots + y_d \Lambda_d) \cdot  V^\top. 
 \end{equation}
 Let $\mathcal{D}$ denote this
   ASSM of  diagonal matrices, i.e.~$\,\mathcal{D} \,= \,
\{ \Lambda_0 + y_1 \Lambda_1 + \cdots + y_d \Lambda_d \,:\, y \in \RR^d \}$.
Then the linear change of coordinates given by $V$ identifies the 
respective Gibbs manifolds:
\begin{equation}
\label{eq:changewithV2}
{\rm GM}(\mathcal{L}) \,  = \, V \cdot {\rm GM}(\mathcal{D}) \cdot V^\top.
\end{equation}
The same statement holds for the Gibbs varieties and the  Gibbs planes:
\begin{equation}
\label{eq:changewithV3}
{\rm GV}(\mathcal{L}) \,  = \, V \cdot {\rm GV}(\mathcal{D}) \cdot V^\top
\qquad {\rm and} \qquad
{\rm GP}(\mathcal{L}) \,  = \, V \cdot {\rm GP}(\mathcal{D}) \cdot V^\top.
\end{equation}
The dimensions of these objects are determined by
arithmetic properties of the eigenvalues.

 Recall that $\Lambda_i = {\rm diag}(\lambda_i)$ where
 $\lambda_i $ is a vector in $ \RR^n$.
Let $\Lambda$ denote the linear subspace of $\RR^n$ that is
spanned by the $d$ vectors $\,\lambda_1,\ldots,\lambda_d$.
We have $\mathcal{D} = \lambda_0 + \Lambda$, and therefore
$$ {\rm GM}(\mathcal{D}) \,\, = \,\,
 {\rm exp}(\lambda_0) \, \star \,  {\rm exp}(\Lambda) \, = \,  \{ (e^{\lambda_{01}} w_1, \ldots, e^{\lambda_{0n}} w_n) \, : \, w \in {\rm exp}(\Lambda) \} \,\,\, \subset \,\,\, \RR^n.
$$
Here $\star$ denotes coordinate-wise multiplication in $\RR^n$.
Let $\Lambda_\QQ$ be the smallest vector subspace of $\RR^n$
spanned by elements of $\QQ^n$ which contains $\Lambda$.
Its dimension $\,d_\QQ = {\rm dim}\,\Lambda_\QQ\,$ satisfies $\, d \leq d_\QQ \leq n$.
Fix lattice vectors $\,a_1,a_2,\ldots,a_{d_\QQ}\,$
in $\ZZ^n$ that form a basis of $\Lambda_\QQ$.
Then, inside an~$n$-dimensional linear space defined by the diagonality condition, we have
\begin{equation}
\label{eq:toric} {\rm GV}(\mathcal{D}) \,\, = \,\,\overline{\biggl\{\bigl(\,
e^{\lambda_{01}} \prod_{i=1}^{d_\QQ} z_i^{a_{i1}},\,\,
e^{\lambda_{02}} \prod_{i=1}^{d_\QQ} z_i^{a_{i2}},\,\ldots\,,\,\,
e^{\lambda_{0n}} \prod_{i=1}^{d_\QQ} z_i^{a_{in}}\, \bigr) \,\,:\,\,
z \in (\CC^{*})^{d_\QQ} \,\biggr\}}. 
\end{equation}
This is a toric variety of dimension  $d_\QQ$.
Just like in \cite[Section~2]{ourteam}, the closure is taken in~$\CC^n$. 
The Gibbs manifold ${\rm GM}(\mathcal{D})$ is a 
$d$-dimensional subset
of the real points in ${\rm GV}(\mathcal{D})$ for which 
$z$ has strictly positive coordinates.
We summarize our discussion in the following theorem.

\begin{theorem}
Let $\mathcal{L}$ be an affine space of pairwise commuting symmetric matrices.
Then the Gibbs variety ${\rm GV}(\mathcal{L})$ is a toric variety of dimension $d_\QQ$,
given explicitly by (\ref{eq:changewithV3}) and
(\ref{eq:toric}).
\end{theorem}

For an illustration, consider the seemingly simple case
  $d=1$ and $A_0 = 0$. Here,
 ${\rm GM}(\mathcal{L})$ is the curve formed by
all powers of ${\rm exp}(A_1)$, and
${\rm GV}(\mathcal{L})$ is a
toric variety of generally higher dimension.
This scenario is reminiscent of that studied  
by Galuppi and Stanojkovski in \cite{GS}.

\begin{example} \label{ex:3commuting}
Let $n=3$ and consider the LSSM $\mathcal{L}$ spanned by
$  A_1 =  \begin{small}
\begin{bmatrix}
4 & 1 & 1 \\
1 & 3 & 1 \\
1 & 1 & 3 \end{bmatrix}\end{small}$.
We have
$$ A_1 \,=\, V \cdot {\rm diag} (\lambda) \cdot V^\top , \,\, {\rm where} \quad
\lambda =  \bigl(2, 4 + \sqrt{2}, 4 - \sqrt{2} \bigr)  \quad  {\rm and} \quad
 V \,=\,  \frac{1}{2} \begin{small} \begin{bmatrix} 0 & \sqrt{2} & - \sqrt{2} \\
- \sqrt{2} & 1 & 1 \\
 \sqrt{2} & 1 & 1 \end{bmatrix}. \end{small} $$
Here, $\mathcal{D} = \Lambda = \RR \lambda$,
$\,d_\QQ = 2$, and 
 $\Lambda_\QQ \,=\,\RR \{ (1,2,2), (0,1,-1) \} \,=\,
\{p \in \RR^3 \,:\, 4 p_1 = p_2 + p_3 \}$.
Hence ${\rm GV}(\mathcal{D})$ is the toric surface 
$\{q_{11}^4 = q_{22} q_{33}\}$ in 
$\,{\rm GP}(\mathcal{D}) = \{ Q \in \mathbb{S}^3\,:\,
q_{12} = q_{13} = q_{23} = 0 \}$.
We transform that surface into the original coordinates
via (\ref{eq:changewithV3}).
  The computation reveals
$$ {\rm GV}(\mathcal{L})  = \{ X \in {\rm GP}(\mathcal{L})
\,:\,x_{23}^4-4 x_{23}^3x_{33}
+6x_{23}^2 x_{33}^2-4 x_{23} x_{33}^3+x_{33}^4+2 x_{13}^2-x_{23}^2-2x_{23}x_{33}-x_{33}^2 =  0 \}.$$
The ambient $3$-space is
$\, {\rm GP}(\mathcal{L}) = \{ X \in \mathbb{S}^3:
 x_{11}-x_{23}-x_{33} = x_{12}-x_{13} =x_{22}-x_{33} = 0\}$.
\end{example}

This concludes our discussion of the toric Gibbs varieties
that arise from pairwise commuting matrices. In the next section
we turn to the general case, which requires new ideas.

\section{Implicitization of Gibbs varieties} \label{sec:3}
Implicitization is the computational problem of finding implicit equations for an object that comes in the form of a parametrization. 
 When the parametrizing functions are rational functions, these equations are polynomials and can be found using resultants or Gr\"obner bases \cite[Section 4.2]{MS}.
   A different approach rests on polynomial interpolation and numerical nonlinear algebra. This section studies the implicitization problem for Gibbs varieties.
The difficulty arises from the fact that 
Gibbs manifolds are transcendental, since
their parametrizations involve the exponential function.
We start out by presenting our proof of Theorem \ref{thm:dimGV}.

As in Section \ref{sec:2}, ${\cal L} = A_0 + {\rm span}_{\mathbb{R}}(A_1, \ldots, A_d)$ is a $d$-dimensional affine space of symmetric $n \times n$-matrices. Its elements are $A_0 + y_1A_1 + \cdots + y_d A_d$. We shall parametrize the Gibbs manifold ${\rm GM}({\cal L})$ in terms of the coordinates $y_1, \ldots, y_d$ on ${\cal L}$. This uses the following formula. 

\begin{theorem}[Sylvester \cite{Sylvester}] \label{thm:sylv}
    Let $f: D \rightarrow \mathbb{R}$ be an analytic function on an open set $D \subset \mathbb{R}$ and~$M \in \mathbb{R}^{n \times n}$ a matrix that has $n$ distinct eigenvalues~$\lambda_1,\ldots,\lambda_n $ in $D$. Then
\[ f(M) \,=\, \sum\limits_{i=1}^{n}f(\lambda_i)M_i, \quad \text{with} \quad M_i \,=\, \prod_{j\neq i}\dfrac{1}{\lambda_i-\lambda_j}(M - \lambda_j  \cdot {\rm id}_n).\]
We note that the product on the right hand side takes place
in the commutative ring $\,\RR[M]$.
\end{theorem}

\begin{proof}[Proof of Theorem \ref{thm:dimGV}]
The characteristic polynomial of $A(y) = A_0 + y_1 A_1 + \cdots + y_d A_d$ equals
    \[ P_{\cal L}(\lambda;y) \, = {\rm det}(A(y) - \lambda \cdot {\rm id}_n) \,=\,
        \, c_0(y) + c_1(y) \, \lambda + \cdots + c_{n-1}(y) \, \lambda^{n-1} + (-1)^n \, \lambda^n. \]
Its zeros $\lambda$  are algebraic functions of the coordinates $y = (y_1, \ldots, y_d)$ on $\mathcal{L}$. 

We first assume that $\mathcal{L}$ has distinct eigenvalues, i.e.~there is a Zariski open subset $U \subset \mathbb{R}^d$ such that  $P_{\cal L}(\lambda;y^*)$ has $n$ distinct real roots $\lambda$
 for all $y^* \in U$.
 Sylvester's formula writes the entries of ${\rm exp}(A(y))$ as rational functions of $y, \lambda_i(y)$ and $e^{\lambda_i(y)}$ for $y \in U$. These functions are multisymmetric in the pairs $(\lambda_i, e^{\lambda_i})$.
 They evaluate to convergent power series on $\mathbb{R}^d$. 

    Let $V$ be the subvariety of $U \times \mathbb{R}^n$ that is     defined by the equations 
    \begin{equation} \label{eq:symmeq}
       c_i(y) \, = \, (-1)^i \sigma_{n-i}(\lambda) \quad {\rm for} \quad  i = 0, \ldots, n-1,
    \end{equation}
    where $\sigma_t(\lambda)$ is the 
    $t^{\rm th}$ elementary symmetric polynomial  evaluated at  $(\lambda_1, \ldots, \lambda_n)$. 
    We have $\dim V = d$. Define a map $\phi: V \times \mathbb{R}^n \rightarrow \mathbb{S}^n$, using coordinates $z_1, \ldots, z_n$ on $\mathbb{R}^n$, as follows:
    \begin{equation} \label{eq:param}
    (y_1, \ldots, y_d, \lambda_1, \ldots, \lambda_n, z_1, \ldots, z_n) \, \longmapsto \, \sum_{i=1}^n z_i \, \prod_{j \neq i} \frac{1}{\lambda_i-\lambda_j} (A(y) - \lambda_j \cdot {\rm id}_n).
    \end{equation}
    The closure  $\overline{\phi(V \times \mathbb{R}^n)}$
    of the image of this map is a variety. It
     contains the Gibbs variety:
     setting $z_i = e^{\lambda_i}$ parametrizes a dense subset of the Gibbs manifold, by Theorem~\ref{thm:sylv}.
       
       The Gibbs variety of the LSSM $\RR {\cal L}$ spanned
       by the ASSM ${\cal L}$ also lies in $\overline{\phi(V \times \mathbb{R}^n)}$,
        because
        ${\rm exp}(y_0 A(y))=\phi(y_0 \cdot y, y_0 \cdot \lambda, e^{y_0 \cdot \lambda})$ for any $y \in U$ and $y_0 \in \mathbb{R} \backslash \{0\}$. We thus have
    \[ \dim {\rm GV}({\cal L}) \,\leq \,\dim \overline{\phi(V \times \mathbb{R}^n)}\, \leq \,d + n
     \quad \text{and} \quad \dim {\rm GV}( \RR {\cal L})\,  \leq \,d + n.\]
     Finally, suppose that ${\cal L}$ is an LSSM, i.e.~$A_0 = 0$.
     Then $\mathcal{L}$ is the linear span of an ASSM of dimension $d-1$ in $\mathbb{S}^n$.
     The  second inequality therefore gives $\,\dim {\rm GV}({\cal L}) \leq d + n -1$.
     
     We finally consider the case when ${\cal L}$ has $m < n$ distinct eigenvalues.
      Since symmetric matrices are diagonalizable, Sylvester's formula can easily be adapted to this case: it suffices to sum over the distinct eigenvalues of $M$, and to adjust the parametrization \eqref{eq:param} accordingly. That is, we replace $n$ by $m$.
      See \cite[Chapter~6.1, Problem 14]{HJ} for details. \end{proof}

\begin{remark} \label{rem:W}
If the points ${\rm exp}(\lambda(y)) = (e^{\lambda_1(y)}, \ldots, e^{\lambda_n(y)})$, $y \in U$,
 lie in a lower-dimensional subvariety $W \subset \mathbb{R}^n$,
then  the proof of Theorem \ref{thm:dimGV} gives the better bound $\dim {\rm GV}({\cal L}) \leq d + \dim W$.
We saw this in Example~\ref{ex:3commuting}.
In general,  no such subvariety $W$ exists, i.e.~one expects
$W = \mathbb{R}^n$.
This is an issue of Galois theory, to be
discussed at the end of this section.

\end{remark}

For ease of exposition, we work only with LSSMs in the rest of this section. That is, we set $A_0 = 0$. We comment on the generalization to ASSMs in Remark \ref{rem:ASSM}. Our discussion and the proof of Theorem \ref{thm:dimGV} suggest
Algorithm \ref{alg:impl}, for computing
the ideal of the Gibbs variety of an LSSM ${\cal L}$. 
\begin{algorithm}[h!]
\caption{Implicitization of the Gibbs variety of an LSSM ${\cal L}$, defined over $\mathbb{Q}$}\label{alg:impl}
 \hspace*{\algorithmicindent} \textbf{Input:} Linearly independent matrices $A_1, \ldots, A_d \in \mathbb{S}^n$ with rational entries \\
 \hspace*{\algorithmicindent} \textbf{Output:} 
 Polynomials that define ${\rm GV}({\cal L})$, where ${\cal L} = {\rm span}_{\mathbb{R}}(A_1, \ldots, A_d)$
\begin{algorithmic}[1]
\State \label{step:charpol} $\text{Compute the characteristic polynomial } \,P_{\cal L}(\lambda;y) = c_0(y) + c_1(y) \lambda + \cdots + c_n(y) \lambda^n$ 
\Require $P_{\cal L}(\lambda;y)$ has $n$ distinct roots in $\overline{\mathbb{R}(y)}$
\State \label{step:vieta} $E_1' \gets \{ \text{the $n$ polynomials $(-1)^i \sigma_{n-i}(\lambda) - c_i(y)$ in \eqref{eq:symmeq}} \} $
\State \label{step:prime} $E_1 \gets \{ \text{generators of any associated prime over } \mathbb{Q} \text{ of }\langle E_1' \rangle \}$
\State \label{step:param} $E_2 \gets \{ \text{the entries of } \phi(y,\lambda,z) - X \}, \text{ with $X = (x_{ij})$ a
symmetric matrix of variables}$
\State \label{step:deno} $E_2, D \gets \text{ clear denominators in $E_2$ and record the least common denominator $D$}$
\If{the roots $\lambda_1, \ldots, \lambda_n$ of $P_{\cal L}(\lambda;y)$ are $\mathbb{Q}$-linearly dependent} \label{step:if}
    \State \label{step:tor} $E_3 \gets \{z^\alpha-  z^\beta \, : \,  \sum \alpha_i \lambda_i = \sum \beta_j \lambda_j , \, \alpha, \beta \in \mathbb{Z}_{\geqslant 0}^n \}$
    \Else \label{step:else} 
    \State \label{step:nothing} $E_3 \gets \emptyset$
\EndIf
\State \label{step:I} $I \gets \text{ the ideal generated by $E_1, E_2, E_3$ in the polynomial ring $\mathbb{R}[y,\lambda,z,X]$}$
\State \label{step:sat} $I \gets I : D^\infty$
\State \label{step:J} $J \gets \text{ elimination ideal obtained by eliminating $y, \lambda, z$ from $I$}$\\
\Return \label{step:return} $\text{a set of generators of } J$
\end{algorithmic}
\end{algorithm}
That ideal lives in a polynomial ring
$\RR[X]$ whose variables
are the entries of a symmetric $n \times n$ matrix.
 The algorithm builds
     three subsets $E_1, E_2, E_3$ of the
    larger polynomial ring $\mathbb{R}[y,\lambda,z, X]$.
    After the saturation (step \ref{step:sat}), the auxiliary~variables
         $y, \lambda, z$ are  eliminated. The equations $E_1'$ come from \eqref{eq:symmeq}. They constrain~$(y,\lambda)$ to lie in $V$. The set $E_1$ generates an associated prime of $\langle E_1' \rangle$ (step 
         \ref{step:prime}), see the discussion preceding Theorem \ref{thm:prime}. The equations $E_2$ come from the parametrization \eqref{eq:param}. Note that, if ${\cal L}$ has $m < n$ distinct eigenvalues, this formula can be adjusted as in the end of the proof of Theorem \ref{thm:dimGV}, and the requirement after step \ref{step:charpol} can be dropped. Later in the algorithm, one replaces $n$ with $m$. It is necessary to clear denominators in order to obtain polynomials (step \ref{step:deno}). The saturation by the LCD $D$ avoids spurious components arising from this step. Finally, $E_3$ accounts for toric relations between the $z_i$
     arising from $\QQ$-linear relations among the $\lambda_i$.
     If no such relations exist, then Theorem \ref{thm:ax} ensures that
     the assignment $E_3 \gets \emptyset$  in step \ref{step:nothing} is correct.

Steps \ref{step:if} and \ref{step:tor} in Algorithm~\ref{alg:impl} require
a detailed discussion. Further below we shall
 explain the 
$\mathbb{Q}$-linear independence of eigenvalues, how to
check this, and how to compute~$E_3$.
Ignoring this for now, one can also run Algorithm~\ref{alg:impl}
with $E_3 = \emptyset$. Then step \ref{step:return} still returns polynomials that vanish
on the Gibbs variety ${\rm GV}(\mathcal{L})$ but these may cut out a larger variety.

We implemented Algorithm \ref{alg:impl} in \texttt{Julia} (v1.8.3), using \texttt{Oscar.jl} \cite{Oscar}, and tested it on many examples. The code is available at \url{https://mathrepo.mis.mpg.de/GibbsManifolds}. 

\begin{example}
    The Gibbs variety ${\rm GV}({\cal L})$ for the LSSM ${\cal L}$ in \eqref{eq:Lintro} has the parametrization
    $$
    \phi \,= \,\,
     \sum_{i=1}^{3}\dfrac{z_i}{q(\lambda_i, y_1, y_2, y_3)} \begin{small}
     \begin{bmatrix}
    p_{11}(\lambda_i, y_1, y_2, y_3)&  p_{12}(\lambda_i, y_1, y_2, y_3) & p_{13}(\lambda_i, y_1, y_2, y_3) \\
     p_{12}(\lambda_i, y_1, y_2, y_3)& p_{22}(\lambda_i, y_1, y_2, y_3) & p_{23}(\lambda_i, y_1, y_2, y_3) \\
    p_{13}(\lambda_i, y_1, y_2, y_3)&  p_{23}(\lambda_i, y_1, y_2, y_3) &  p_{33}(\lambda_i, y_1, y_2, y_3) \\
    \end{bmatrix} \end{small}
    , \,\,\, {\rm where} $$
    $$ \begin{small}
    \begin{matrix}
        & q &=& 2y_1^2 + 6y_1y_2 + 2y_2^2 + 6y_1y_3 + 6y_2y_3 + 2y_3^2 - 6y_1\lambda - 6y_2\lambda - 6y_3\lambda + 3\lambda^2,\\
        & p_{11} &=&y_1^2 + 2y_1y_2 + y_2^2 + 2y_1y_3 + 2y_2y_3 - 2y_1\lambda - 2y_2\lambda - 2y_3\lambda + \lambda^2,\\
        & p_{12} &=& -y_1^2 - y_1y_2 - y_1y_3 + y_2y_3 + y_1\lambda,\\
        & p_{13} &=& -y_1y_2 - y_2^2 + y_1y_3 - y_2y_3 + y_2\lambda,\\
        & p_{22} &=& y_1^2 + 2y_1y_2 + 2y_1y_3 + 2y_2y_3 + y_3^2 - 2y_1\lambda - 2y_2\lambda - 2 y_3\lambda + \lambda^2,\\
        & p_{23} &=& y_1 y_2 - y_1y_3 - y_2y_3 - y_3^2 + y_3\lambda,\\
        & p_{33} &=& 2y_1y_2 + y_2^2 + 2y_1y_3 + 2y_2y_3 + y_3^2 - 2y_1\lambda - 2 y_2\lambda - 2y_3\lambda + \lambda^2.
    \end{matrix} \end{small}
    $$
    Our \texttt{Julia} code for
     Algorithm \ref{alg:impl} easily finds
    the cubic polynomial  defining ${\rm GV}(\mathcal{L)}$.
\end{example}

In spite of such successes, symbolic implicitization 
   is limited to small $n$ and $d$. Numerical computations can help, in some cases, to find 
   equations for more challenging Gibbs varieties.

\begin{example} \label{ex:gram}
    We consider the LSSM of $4 \times 4$ Hankel matrices with upper left entry zero: 
    \[ {\cal L} \, = \, \left\{ \begin{bmatrix}
        0 & y_2 & y_3 & y_4 \\ y_2 & y_3 & y_4 & y_5\\ y_3 & y_4 & y_5 & y_6\\ y_4 & y_5 & y_6 & y_7
    \end{bmatrix}  \, : \, (y_2, \ldots, y_7) \in \mathbb{R}^6 \right\}.\]
    Algorithm \ref{alg:impl} failed to compute its Gibbs variety.
     We proceed using numerics as follows. Fix a degree $D > 0$ and let $N = \binom{9 + D }{D}$ be the number of monomials in the 10 coordinates $x_{11}, \ldots, x_{44}$ on $\mathbb{S}^4$. We create $M \geq N$ samples on ${\rm GM}({\cal L})$ by plugging in random values for the six parameters $y_i$ and applying the matrix exponential. Finding all vanishing equations of degree $D$ on these samples amounts to computing the kernel
    of an $M \times N$ Vandermonde matrix. If this matrix has full rank, then
    there are no relations of degree $D$. We implemented this procedure in \texttt{Julia}. 
    In our example, Theorem \ref{thm:dimGV} suggests that ${\rm GV}({\cal L})$ is a hypersurface,
    and this is indeed the case. Its defining equation has degree $D = 6$.   We found it using $M = 5205 \geq N = 5005$ samples. Our denary sextic has $853$ terms with integer coefficients:
$$
x_{11}^3 x_{22} x_{24} x_{34}-x_{11}^3 x_{23}^2 x_{34}
-x_{11}^3 x_{23} x_{24}^2+x_{11}^3 x_{23} x_{24} x_{33}
+ \cdots + 
3 x_{23} x_{24}^2 x_{33} x_{34}^2+x_{24}^4 x_{33} x_{34}-x_{24}^3 x_{33}^2 x_{34}.
$$
Its Newton polytope has the f-vector
$(456,
5538,
21560,
41172,
44707,
29088,
11236,
2370,
211)$.
    Note that the package \texttt{Oscar.jl} conveniently allows to perform symbolic and numerical implicitization and polyhedral computations in the same programming environment.  
    
    We emphasize that our numerical
    {\tt Julia} code is set up to find \emph{exact} integer coefficients. For this, we first normalize the numerical approximation of the coefficient vector by setting its first (numerically) nonzero entry to one. Then we rationalize the coefficients using the built in command \texttt{rationalize} in \texttt{Julia}, with error tolerance \texttt{tol = 1e-7}. Correctness of the result is proved by checking that the resulting polynomial vanishes on the parametrization. 
\end{example}

We now turn to 
$\QQ$-linear relations among eigenvalues of $\mathcal{L}$.
Our arithmetic discussion begins with a version of
\cite[(SP)]{Ax}, which is well-known
in transcendental number~theory:

\begin{theorem}[Ax-Schanuel] \label{thm:ax}
If the eigenvalues $\lambda_1,\ldots,\lambda_n$ of the LSSM $\,\mathcal{L}$
are $\mathbb{Q}$-linearly independent, then
 $e^{\lambda_1}, \ldots, e^{\lambda_n}$ are algebraically independent over the field
  $\,\mathbb{C}(y_1,\ldots,y_d)$.
 \end{theorem}

On the other hand, suppose that the eigenvalues 
of $\mathcal{L}$ satisfy some non-trivial linear relation
over $\mathbb{Q}$. We can then find nonnegative integers
$\alpha_i$ and $\beta_j$, not all zero, such that
\begin{equation} \label{eq:linrel}
 \sum_{i=1}^n \alpha_i \lambda_i \,\,=\,\, \sum_{j=1}^n \beta_j \lambda_j. 
 \end{equation}
This implies that the exponentials of the eigenvalues satisfy the toric relations
\begin{equation}
\label{eq:torrel}
 \prod_{i=1}^n z_i^{\alpha_i} \,\, = \,\, \prod_{j=1}^n z_j^{\beta_j} .
 \end{equation}

The linear relations (\ref{eq:linrel}) can be found
from the ideal $\langle E_1' \rangle$ in step \ref{step:vieta} which 
specifies that the $\lambda_i$ are the eigenvalues of $A(y)$.
This ideal is radical if we assume that $\mathcal{L}$ has distinct eigenvalues.
We compute the prime decomposition of the ideal over $\mathbb{Q}$.
All prime components are equivalent under permuting the $\lambda_i$,
so we replace $\langle E_1' \rangle$ by any of these prime ideals in step \ref{step:prime}.
We compute (\ref{eq:linrel}) as the linear forms in that prime ideal.
Using (\ref{eq:torrel}), we compute the toric ideal $\langle E_3 \rangle$ in step \ref{step:tor}, which is also prime. This ideal defines a toric variety $W'$, whose $S_n$-orbit is the variety $W$ in Remark \ref{rem:W}.
We arrive at the following~result.

\begin{theorem} \label{thm:prime}
Let $\mathcal{L} \subset \mathbb{S}^n$ be an LSSM 
with  distinct eigenvalues. 
The Gibbs variety~${\rm GV}(\mathcal{L})$ is irreducible and unirational, and the ideal $J$ found in
Algorithm \ref{alg:impl} is its
prime ideal.
\end{theorem}

\begin{proof}
Sylvester's formula yields a rational parametrization $\psi$ of ${\rm GV}(\mathcal{L})$ 
with parameters $y_1,\ldots,y_d,z_1,\ldots,z_n$. The parameters $\lambda_i$ in \eqref{eq:param} can be omitted: the entries in the image are multisymmetric in $(\lambda_i,z_i)$, so that they can be expressed in terms of elementary symmetric polynomials of the $\lambda_i$ \cite[Theorem 1]{Briand}. The point $(z_1, \ldots, z_n)$ lies on the toric variety $W'$ defined above. The domain $\mathbb{C}^d \times W'$ of $\psi$ is an irreducible
variety, and it is also rational.
The image of $\psi$ is the Gibbs variety ${\rm GV}(\mathcal{L})$, which is therefore
unirational and irreducible. The ideals given by $E_1$ and $E_2$
in Algorithm~\ref{alg:impl} are prime, after saturation, and
elimination in step \ref{step:J} preserves primality. Hence the output in
$J$ in step \ref{step:return} is the desired prime ideal.
\end{proof}

We define the {\em Galois group} $G_{\mathcal{L}}$ of an LSSM $\mathcal{L}$
to be the Galois group of the characteristic polynomial
$P_\mathcal{L}(\lambda,y)$ over the field $\mathbb{Q}(y_1,\ldots,y_d)$.
Note that $G_\mathcal{L}$ is the subgroup of the symmetric group $S_n$ 
whose elements are
 permutations that fix each associated prime of $\langle E_1' \rangle$.
 Hence the index of the Galois group $G_\mathcal{L}$ in $S_n$
is the number of associated primes.
In particular, the Galois group equals $S_n$ if and only if 
the ideal $\langle E_1' \rangle$ formed in step \ref{step:vieta} is prime.

The existence of linear relations 
(\ref{eq:linrel}) depends on the Galois group 
$G_\mathcal{L}$. If the Galois group is small
then the primes of $\langle E_1 \rangle$ are large,
and more likely to contain linear forms. There is
a substantial literature in number theory on this topic.
See \cite{Gir1, Gir2} and the references therein.
For instance, by
Kitaoka \cite[Proposition 2]{Kit},
there are no linear relations if
$n$ is prime, or if $n \geq 6$ and the
Galois group is $S_n$ or $A_n$. If this holds, 
$E_3 = \emptyset$ in step \ref{step:nothing} of Algorithm~\ref{alg:impl}.

The computation of Galois groups is a well-studied
topic in symbolic computation and number theory.
 Especially promising are methods based on
 numerical algebraic geometry (e.g.~in \cite{HRS}).
  These fit well with the  approach to implicitization in
   Example \ref{ex:gram}. For a future theoretical project,
it would  be very interesting to classify LSSMs by their Galois groups.

\begin{remark} \label{rem:ASSM}
    We briefly comment on how to adjust Algorithm \ref{alg:impl} to compute the Gibbs variety of an ASSM ${\cal L}$ with $A_0 \neq 0$. In this case, algebraic relations between $e^{\lambda_1}, \ldots, e^{\lambda_n}$ come from $\mathbb{Q}$-linear relations between the eigenvalues of ${\cal L}$, but this time modulo $\mathbb{C}$: an affine relation $\sum \alpha_i \lambda_i = \sum \beta_j \lambda_j + \gamma$ gives $z^\alpha - e^\gamma \cdot z^\beta = 0$, where $z_i = e^{\lambda_i}$, $\alpha_i, \beta_j \in \mathbb{Z}_{\geq 0}$, $\gamma \in \mathbb{C}$. Here $\gamma$ is a $\mathbb{Q}$-linear combination of eigenvalues of $A_0$. Theorem \ref{thm:prime} holds for ASSMs as well, provided that these $\mathbb{Q}$-linear relations modulo $\mathbb{C}$ can be computed in practice. This can usually not be done over $\mathbb{Q}$. We leave this algorithmic challenge for future research. 
\end{remark}

\section{Pencils of quadrics} \label{sec:4}

In this section we study the Gibbs variety 
  ${\rm GV}(\mathcal{L})$
where $\mathcal{L} \subset \mathbb{S}^n$ is a pencil of quadrics, i.e.~an
LSSM of dimension $d=2$.  We follow the exposition in
  \cite{FMS}, where pencils $\mathcal{L}$ are classified
 by Segre symbols.
 The {\em Segre symbol} $\,\sigma=\sigma(\mathcal{L})$
  is a multiset of partitions that sum up to~$n$. 
  It is computed as follows:
    Pick a basis~$\{A_1, A_2\}$ of~$\mathcal{L}$, where~$A_2$ is invertible, and find the Jordan canonical form of~$A_1 A_2^{-1}$. Each eigenvalue determines a partition, according to the sizes of the corresponding Jordan blocks. The multiset of these partitions is the Segre symbol $\sigma$.

 We use the  canonical form 
 in \cite[Section 2]{FMS}.
Suppose the Segre symbol is $\sigma = [\sigma_1,\ldots,\sigma_r]$,
where the $i$th partition
$\sigma_{i} $ equals $ (\sigma_{i,1} \geq \sigma_{i,2} \geq \cdots \geq \sigma_{i,n} \geq 0)$.
There are $r$ groups of blocks, one for each eigenvalue $\alpha_i$ of $A_1 A_2^{-1}$.
The $j$th matrix in the $i$th group 
 is the $ \sigma_{i,j} \times \sigma_{i,j}$ matrix
$$
y_1 \cdot \begin{small} \begin{bmatrix} 
0 & 0 & \ldots & 0 & \alpha_i \\
0 & 0 & \ldots & \alpha_i & 1 \\
\vdots & \vdots & \Ddots & \Ddots & \vdots\\
0 & \alpha_i & 1 & \Ddots & 0\\
\alpha_i & 1 & \ldots & 0 & 0 \\
\end{bmatrix} \end{small} \,+ \,\,y_2 \cdot
\begin{small}
\begin{bmatrix} 
0 & \ldots & 0 & 0 & 1\\
0 & \ldots & 0 & 1 & 0\\
0 & \ldots & 1 & 0 & 0\\
\vdots & \Ddots & \vdots & \vdots & \vdots\\
1& \ldots & 0 & 0 & 0\\
\end{bmatrix} \end{small}
.$$

There are $13$ Segre symbols for $n=4$; see \cite[Example 3.1]{FMS}.
It is instructive to compute their Gibbs varieties. 
All possible dimensions, $2,3,4$ and $5$, are attained.
Dimension $2$ arises for the diagonal pencil
$\mathcal{L}_\sigma =  {\rm diag}(\alpha_1 y_1{+}y_2, \alpha_2 y_1 {+}  y_2 , \alpha_3 y_1 {+}  y_2, 
\alpha_4 y_1{+}  y_2)$,
 with Segre symbol  $\sigma = [1,1,1,1]$. When the $\alpha_i$ are distinct integers,
 ${\rm GV}(\mathcal{L}_\sigma) = {\rm GM}(\mathcal{L}_\sigma)$ is a toric surface.
 This is similar to (\ref{eq:toricsurface}).
 Dimension $5$ arises for $\sigma = [4]$, which was presented in Example \ref{ex:n4d2}.
 
The following examples, also computed with
Algorithm \ref{alg:impl},
 exhibit the dimensions $5,4,3$.

\begin{example} \label{ex:dreieins}
Consider the Segre symbol~$\sigma = [3,1]$. The canonical pencil ${\cal L}_{[3,1]}$ is spanned~by
$$
\begin{bmatrix}
0 & 0 & \alpha_1 & 0\\
0 &  \alpha_1 & 1 & 0\\
\alpha_1 & 1 & 0 & 0\\
0 & 0 & 0 &  \alpha_2 \\
\end{bmatrix}
\quad \text{and}\quad
\begin{bmatrix}
0 & 0 & 1 & 0\\
0 & 1 & 0 & 0\\
1 & 0 & 0 & 0\\
0 & 0 & 0 & 1\\
\end{bmatrix}
, \qquad {\rm for} \,\, 
\,\,\alpha_1,\alpha_2 \in \RR \,\,\hbox{distinct}.
$$
Here, $\dim {\rm GV}({\cal L}_{[3,1]}) = 5$, the upper bound in Theorem \ref{thm:dimGV}. 
 Algorithm \ref{alg:impl} produces the ideal
$$ J \, = \,\bigl\langle\,
  x_{14},x_{24},x_{34},\,x_{13}-x_{22}+x_{33},\,x_{12}^2-x_{11}x_{22}-x_{12}x_{23}+x_{11}x_{33}+x_{22}x_{33}-x_{33}^2 \,\bigr\rangle .$$
     If~$\alpha_1= \alpha_2$, then the Segre symbol changes to $\sigma = [(3,1)]$.
     We now find the additional cubic
     \begin{equation}
     \label{eq:additionalcubic}
  x_{11}x_{22}x_{33}+ 2x_{12}x_{13}x_{23} - x_{13}^2 x_{22} - x_{11}x_{23}^2 - x_{12}^2 x_{33} \,-\, x_{44} 
  \quad \in \,\,\, J. \end{equation}
  This cuts down the dimension by one, and we now have
   $\dim {\rm GV}({\cal L}_{[(3,1)]}) = 4$.
\end{example}

\begin{example}
Consider the Segre symbol~$\sigma = [(2,2)]$. The pencil
$\mathcal{L}_{[(2,2)]}$ is spanned by
$$
\begin{bmatrix}
0 & \alpha & 0 & 0\\
\alpha & 1 & 0 & 0\\
0 & 0 & 0 &  \alpha \\
0 & 0 & \alpha & 1\\
\end{bmatrix}
\quad \text{and} \quad
\begin{bmatrix}
0 & 1 & 0 & 0\\
1 & 0 & 0 & 0\\
0 & 0 & 0 & 1\\
0 & 0 & 1 & 0\\
\end{bmatrix},  \qquad \hbox{for some} \,\, 
\,\,\alpha \in \RR. \,\
$$
 A version of Algorithm \ref{alg:impl} for LSSMs with multiple eigenvalues produces the ideal
$$  J \, = \, \langle \,
x_{11}-x_{33},\,x_{12}-x_{34},\,x_{22} - x_{44},\, x_{13}, \,
x_{14},\, x_{23} ,\,x_{24}\,\rangle.$$
 The Gibbs variety ${\rm GV}({\cal L}_{[(2,2)]})$
 is $3$-dimensional and equals the Gibbs plane
   $ {\rm GP}({\cal L}_{[(2,2)]})$.
\end{example}

The cubic (\ref{eq:additionalcubic}) which distinguishes the
Segre symbols $[3,1]$ and $[(3,1)]$
is explained by the following result.
This applies not just to pencils but to all
ASSMs with block structure.
 
 \begin{proposition} \label{prop:blockdiag}
     Let ${\cal L}$ be a block-diagonal ASSM with $r$ blocks $X_i(y)$ of size $\tau_i$, where $\tau_1 + \cdots + \tau_r = n$. The Gibbs plane ${\rm GP}({\cal L})$ is contained in $\mathbb{S}^{\tau_1} \times \cdots \times \mathbb{S}^{\tau_r} \subset \mathbb{S}^n$. Moreover, with the notation ${\cal J} = \{ \{i,j\} \in \binom{[r]}{2}  :  {\rm trace}(X_i(y)) = {\rm trace}(X_j(y)) \}$, we have 
     \[ {\rm GV}({\cal L}) \, \subseteq \, \{ (X_1, \ldots, X_r) \in {\rm GP}({\cal L}) \, : \, 
     \det(X_i) = \det(X_j) \text{ for all } \,\{i,j \} \in {\cal J} \}. \]
 \end{proposition}
 
\begin{proof}
    Block-diagonal matrices are exponentiated block-wise. 
The entries outside~the diagonal blocks are zero. The statement follows from 
$\,{\rm det}({\rm exp}(X_i(y))) = {\rm exp}({\rm trace}(X_i(y)))$.
\end{proof}

 Proposition \ref{prop:blockdiag} holds for the canonical pencil
$\mathcal{L}_\sigma$ of any Segre symbol $\sigma$.
First of all, for all indices $(i,j)$ outside the diagonal blocks, we have
$x_{ij} = 0$ on the Gibbs plane ${\rm GP}({\cal L}_\sigma)$.
 Next, one has equations for 
 the exponential of a single block, like those in Theorem \ref{thm:[n]} below.
 Finally, there are equations that link the blocks corresponding to entries
  $\sigma_{ij}$ of the same partition~$\sigma_i$. 
  Some of these  come from trace equalities between blocks of ${\cal L}_\sigma$,
  and this is the scope of Proposition \ref{prop:blockdiag}.
   In particular, blocks $ij$ and $ik$ for which $\sigma_{ij} = \sigma_{ik} \, {\rm mod} \, 2$ exponentiate to $X_{ij} \in \mathbb{S}^{\sigma_{ij}}_+$ and $X_{ik} \in \mathbb{S}^{\sigma_{ik}}_+$ with equal determinant. 
   We saw this in (\ref{eq:additionalcubic}). In
all examples we computed,
the three classes of equations above determine the
      Gibbs variety.

We now derive the  equations that hold for
 the exponential of a single block. To this end, we fix 
  $\sigma = [n]$ with $\alpha_1 = 0$.
The canonical LSSM~${\cal L}_{[n]}$  consists of 
the symmetric matrices 
$$ Y \,\, = \,\,
\begin{bmatrix}
0 & 0 & \ldots & 0 & y_1\\
0 & 0 & \ldots & y_1 & y_2\\
\vdots & \vdots & \Ddots & \Ddots & \vdots\\
0 & y_1 & y_2 & \vdots & 0\\
y_1 & y_2 & \vdots & 0 & 0\\
\end{bmatrix}.
$$
The case $n=4$ was featured in Example \ref{ex:n4d2}.
In what follows we generalize that example.

\begin{theorem} \label{thm:[n]}
The following linear equations hold on the Gibbs plane~${\rm GP}(\mathcal{L}_{[n]})$:
\begin{equation}
\label{eq:twotwo} x_{i-1,j}+x_{i+1,j}\,\,=\,\,
 x_{i,j-1}+x_{i,j+1}  \quad \hbox{for~$\,2\leqslant i < j \leqslant n$.}  
\end{equation}
The $2 \times 2$-minors of the
 following~$2\times (n-1)$ matrix vanish on the
Gibbs variety ${\rm GV}(\mathcal{L}_{[n]})$:
\begin{equation}
\label{eq:matrixminus}
D(X) \,\, = \,\,\,
\begin{bmatrix}
x_{11} & x_{12} & x_{22} &\ldots&\\
x_{12} & x_{22} & x_{23} &\ldots &\\
\end{bmatrix}
-
\begin{bmatrix}
x_{n,n} & x_{n-1,n} & x_{n-1,n-1}&\ldots &\\
0 & x_{n,n} & x_{n-1,n}& \ldots &\\
\end{bmatrix}.
\end{equation}
If the Galois group $G_{\mathcal{L}_{[n]}}$ is the
symmetric group $ S_n$, then
the prime ideal of ${\rm GV}(\mathcal{L}_{[n]})$
is generated by
(\ref{eq:twotwo}) and~(\ref{eq:matrixminus}), and we have
${\rm dim}\,{\rm GP}(\mathcal{L}_n) = 2n-1$
and  ${\rm dim}\,{\rm GV}(\mathcal{L}_{[n]}) = n+1$.
 \end{theorem}

\begin{remark}
We conjecture that $G_{\mathcal{L}_{[n]}} = S_n$.
This was verified computationally for many values of $n$,
but we currently do not have a proof that works for all $n$. This gap underscores
the need, pointed out at the end of Section \ref{sec:3},
 for a study of the Galois groups of LSSMs.
\end{remark}

\begin{proof}
We claim that the linear equations (\ref{eq:twotwo}) hold for
every non-negative integer power of $Y$. This implies that they hold for ${\rm exp}(Y)$.
We will show this by induction. The equations clearly hold for $Y^0 = {\rm id}_n$.
Suppose they hold for $(m_{ij}) = M = Y^{k}$. Write~$(b_{ij})=B := Y^{k+1} = M Y$. 

The two-banded structure of $Y$ implies
$b_{i,j} = y_1\cdot m_{i,n-j+1} + y_2\cdot m_{i,n-j+2}$ for~$1 \leqslant i < j$. 
The following identity holds for~$2 \leqslant i < j$, and it shows that
${\rm exp}(Y)$ satisfies the 
equations~(\ref{eq:twotwo}):
\begin{multline*}
b_{i-1,j}-b_{i,j-1}-b_{i,j+1}+b_{i+1,j} \,\,=\,\, y_1\cdot m_{i-1,n-j+1} + y_2\cdot m_{i-1,n-j+2} - y_1\cdot m_{i,n-j+2} \, -\\ y_2\cdot m_{i,n-j+3} - y_1\cdot m_{i,n-j} - y_2\cdot m_{i,n-j+1} + y_1\cdot m_{i+1, n-j+1} + y_2\cdot m_{i+1, n-j+2} \\ 
=\,\,y_1\cdot (m_{i-1,n-j+1}-m_{i,n-j+2}-m_{i, n-j}+m_{i+1,n-j+1}) \,  + \\ y_2\cdot(m_{i-1,n-j+2} - m_{i,n-j+3} - m_{i,n-j+1} + m_{i+1,n-j+2}) \,\,=\,\, 0.
\end{multline*}

We next consider the matrix $D(X)$ in (\ref{eq:matrixminus}).
We must show that $D(X)$ has rank $\leq 1$ for $ X \in {\rm GV}(\mathcal{L}_{[n]})$.
We claim that the rows of~$D(Y^k)$ are proportional with the same coefficient for all~$k \in \mathbb{Z}_{\geq 0}$. This will imply that the rows of~$D(\exp{(Y}))$ are proportional.
For the proof,  let~$\vec{v}_1$ and~$\vec{v}_2$ be the rows of~$D(B)$, where~$B = Y^k$. We will show that~${y_1 \vec{v}_1 + y_2 \vec{v}_2 = 0}$.

First note that~$D({\rm id}_n) = 0$. Also note that each column of~$D(B)$ has the form
$$ \begin{bmatrix}
b_{i,i}-b_{n+1-i,n+1-i}\\
b_{i,i+1} - b_{n+1-i,n+2-i}\\
\end{bmatrix} \quad {\rm or} \quad
\begin{bmatrix}
b_{i,i+1} - b_{n-i,n+1-i}\\
b_{i+1,i+1} - b_{n-i+1,n-i+1}\\
\end{bmatrix}.
$$
We start with the left case. We must show~$y_1(b_{i,i}-b_{n+1-i,n+1-i}) + y_2(b_{i,i+1} - b_{n+1-i,n+2-i}) = 0$.

Recall from above  that~$b_{i,j}=y_1\cdot m_{i,n-j+1}+y_2\cdot m_{i,n-j+2}$, where~$(m_{i,j}) = M = Y^{k-1}$ for~$i < j$. Using this and the fact that the powers of~$Y$ are symmetric, we write
\begin{multline*}
    y_1(b_{i,i}-b_{n+1-i,n+1-i}) + y_2(b_{i,i+1} - b_{n+1-i,n+2-i}) =\\= y_1((y_1\cdot m_{i,n-i+1} + y_2\cdot m_{i,n-i+2})-(y_1\cdot m_{n+1-i,i} + y_2\cdot m_{n+1-i,i+1})) \, +\\  y_2((y_1\cdot m_{i,n-i} + y_2 \cdot m_{i,n-i+1})-(y_1\cdot m_{n+1-i,i-1}+y_2\cdot m_{n+1-i,i})) \\= y_1 y_2(m_{i,n-i+2}-m_{i+1,n+1-i}+m_{i,n-i} -m_{n+1-i,i-1}) = 0,
\end{multline*}
where the last equality follows from (\ref{eq:twotwo}).
Now, for the second case we have
\begin{multline*}
    y_1(b_{i,i+1} - b_{n-i,n+1-i}) + y_2(b_{i+1,i+1} - b_{n-i+1,n-i+1}) = \\ y_1(y_1\cdot m_{i,n-i} + y_2\cdot m_{i,n-i+1} - y_1\cdot m_{n-i,i} - y_2\cdot m_{n-i,i+1}) \, +\\  y_2(y_1\cdot m_{i+1,n-i} + y_2\cdot m_{i+1,n-i+1}-y_1\cdot m_{n-i+1,i} - y_2\cdot m_{n-i+1,i+1}) = 0.
\end{multline*}
This proves that the $2 \times 2$ minors of $D(X)$ vanish on 
the Gibbs variety $\,{\rm GV}(\mathcal{L}_{[n]} )$.

Suppose now that the eigenvalues of $Y$ are $\QQ$-linearly independent.
We can check this directly for $n \leq 5$. For $n \geq 6$ it follows 
from our hypothesis $G_{\mathcal{L}_{[n]}} = S_n$, by  \cite[Proposition~2]{Kit}.
That hypothesis implies ${\rm dim} \,{\rm GV}(\mathcal{L}_{[n]} )= n+1$, by Theorems \ref{thm:dimGV}
and \ref{thm:ax}. The $2 \times 2$-minors of $D(X)$
generate a prime ideal of codimension $n-2$ in the coordinates
of the $(2n-1)$-dimensional space given by (\ref{eq:twotwo}).
The equality of dimensions  yields
 ${\rm GP}(\mathcal{L}_{[n]}) = 2n-1$, and we  conclude
that our linear and quadratic constraints generate
the prime ideal of ${\rm GV}(\mathcal{L}_{[n]})$.
\end{proof}

Theorem \ref{thm:dimGV}
and its refinement in Remark \ref{rem:W}
furnish an upper bound on the dimension of
any Gibbs variety. This raises the question
when this bound is attained. In what follows, we offer a complete answer for 
$d=2$. Let $\mathcal{L}$ be a pencil 
with eigenvalues $\lambda_i(y)$, and let
$W$ denote the Zariski closure in $\RR^n$
of the set of points ${\rm exp}(\lambda(y)) = (e^{\lambda_1(y)}, \ldots, e^{\lambda_n(y)})$,
$y \in \RR^2$.

\begin{theorem}
Let $\mathcal{L} = {\rm span}_{\mathbb{R}}(A_1,A_2)$, where
$A_1 A_2 \not= A_2 A_1$. Then
${\rm dim} \,{\rm GV}(\mathcal{L}) = {\rm dim}(W)+1$.
In particular, if the Galois group $G_{\mathcal{L}} $ is the symmetric group $ S_n$ then
${\rm dim} \,{\rm GV}(\mathcal{L}) = n+1$.
\end{theorem}

\begin{proof}
We claim that the fibers of the map $\phi: V\times W \to \mathbb{S}^n$ defined by \eqref{eq:param}~are one-dimensional. Let~$B \in \phi(V \times W)$
and consider any point $p=(y_1,y_2,\lambda_1,\ldots,\lambda_n,z_1,\ldots,z_n) \in \phi^{-1}(B)$.
The condition that $p$ lies in the fiber $\phi^{-1}(B)$ is equivalent to 
\begin{enumerate}
    \item[(1)] $z_1,\ldots, z_n$ are the eigenvalues of $B$, and
    \item[(2)] $X = y_1A_1+y_2A_2$ and $B$ have the same eigenvectors, and
    \item[(3)] $\lambda_1, \ldots, \lambda_n$ are the eigenvalues of $X$.
\end{enumerate}
Condition (1) follows from Theorem \ref{thm:sylv} for
$f = {\rm exp}$. It implies that there
are only finitely many possibilities for the $z$-coordinates of 
the point $p$ in the fiber: up to permutations, they are the eigenvalues of $B$.
Condition (3) follows from~$(y_1,y_2,\lambda_1,\ldots,\lambda_n)\in V$. It says that the $\lambda$-coordinates are determined, up to permutation, by $y_1,y_2$. Therefore, it suffices to show that the matrices in ${\cal L}$ whose eigenvectors are those of $B$ form a one-dimensional subvariety.

Symmetric matrices have common eigenvectors if and only if they commute. Define $S = \{ X = y_1 A_1 + y_2 A_2 \in {\cal L} \, : X \cdot B = B \cdot X \} \subset {\cal L}$. This is a pairwise commuting linear subspace. Note that $S$ contains a nonzero matrix $X$, since there is a point in $\phi^{-1}(B)$ whose $y$-coordinates define a nonzero matrix in ${\cal L}$. Therefore $\dim S \geq 1$. Since $A_1A_2 \neq A_2A_1$, we also have $\dim S \leq 1$. Hence $\dim S = \dim \phi^{-1}(B) = 1$ and the upper bound~$\dim W + 1$ for the dimension of
${\rm GV}({\cal L})$,
which is given by Remark \ref{rem:W}, is attained in our situation.
\end{proof}

\section{Convex optimization} \label{sec:5}

In this section we show how Gibbs manifolds arise from
entropic regularization in optimization (cf.~\cite{ourteam}).
We fix an arbitrary linear map $\pi : \mathbb{S}^n \rightarrow \RR^d$.
This can be written in the~form $$ \pi(X) = \bigl( \langle A_1 , X \rangle,
\langle A_2 , X \rangle, \ldots, \langle A_d , X \rangle \bigr). $$
Here the $A_i \in \mathbb{S}^n$, and $\langle A_i , X \rangle := {\rm trace}(A_i X)$.
The image $\pi(\mathbb{S}^n_+)$  of the PSD cone $\mathbb{S}^n_+$
under our linear map $\pi$ is a \emph{spectrahedral shadow}.
Here it is a full-dimensional semialgebraic convex cone in $\RR^d$.  
 Interestingly,
  $\pi(\mathbb{S}^n_+)$
 can fail to be closed, as explained in \cite{Yuhan}.

{\em Semidefinite programming (SDP)} is the 
following convex optimization problem: 
\begin{equation} \label{eq:SDP}
    {\rm Minimize} \quad \langle C , X \rangle \quad  \hbox{subject to} \quad X \in \mathbb{S}^n_+ 
    \,\text{ and }\, \pi(X) = b. 
\end{equation}
See e.g.~\cite[Chapter 12]{MS}.
The instance (\ref{eq:SDP}) is
specified by the cost matrix $C \in \mathbb{S}^n$  and the right hand side vector $b \in \RR^d$. The feasible region $\mathbb{S}^n_+ \cap \pi^{-1}(b)$ is a \emph{spectrahedron}.
The SDP problem \eqref{eq:SDP} is feasible, i.e.~the spectrahedron is non-empty, if and only if $\,b $ is 
in~$ \pi(\mathbb{S}^n_+)$.

Consider the LSSM $\mathcal{L} = {\rm span}_\RR (A_1, \ldots, A_d)$.
We usually assume that $\mathcal{L}$ contains a positive definite matrix.
This hypothesis ensures that each spectrahedron $\pi^{-1}(b)$ is compact.

As a natural extension of
\cite[eqn (2)]{ourteam},
we  now define the entropic regularization of SDP:
 \begin{equation}
\label{eq:regSDP} {\rm Minimize} \quad \langle C , X \rangle \,-\, \epsilon \cdot h(X) \quad  \hbox{subject to} \quad X \in \mathbb{S}^n_+ \,\text{ and } \,\pi(X) = b. 
\end{equation}
Here $\epsilon > 0$ is a small parameter, and $h$ denotes the {\em von Neumann entropy}, here defined as
$$ h \,: \, \mathbb{S}^n_+ \,\rightarrow\, \RR \,, \,\,
X \,\mapsto\, {\rm trace} \bigl( X - X \cdot {\rm log}(X)  \bigr) . $$
We note that $h$ is invariant under the action of the orthogonal group on $\mathbb{S}^n_+$.
This implies $h(X) = \sum_{i=1}^n (\lambda_i-\lambda_i {\rm log}(\lambda_i)  )$,
where $\lambda_1,\ldots,\lambda_n$ are the eigenvalues of $X$.
Hence the von Neumann entropy $h$ is  the matrix version of
 the entropy function on $\RR^n_+$ used in~\cite{ourteam}. 
 
 Our next result makes the role of Gibbs manifolds in semidefinite programming explicit.
The following ASSM is obtained by incorporating $\epsilon$ and the cost matrix $C$ into
the LSSM:
 $$ \mathcal{L}_\epsilon \,\, := \,\, \mathcal{L}- \frac{1}{\epsilon} C \quad
 \hbox{for any} \,\,\epsilon > 0. $$
Here we allow the case $\epsilon = \infty$, where 
the dependency on $C$ disappears and
the ASSM is simply the LSSM, i.e.~$\mathcal{L}_\infty = \mathcal{L}$.
The following theorem is the main result in this section.

\begin{theorem} \label{thm:opt}
   For $b \in \pi(\mathbb{S}^n_+)$, the intersection of $\pi^{-1}(b)$ with the 
   Gibbs manifold ${\rm GM}(\mathcal{L}_\epsilon)$
   consists of a single point $X_\epsilon^*$. This point is the optimal solution to the regularized SDP (\ref{eq:regSDP}).
   For $\epsilon = \infty$, it is the unique
     maximizer of  von Neumann entropy on
     the spectrahedron $\pi^{-1}(b)$.
\end{theorem}

The importance of this result for semidefinite programming lies in taking the
limit as $\epsilon$ tends to zero. This limit $\,{\rm lim}_{\epsilon \rightarrow 0} \,X^*_\epsilon \,$
exists and it is an optimal solution to  (\ref{eq:SDP}). The optimal solution is unique
for generic $C$. Entropic regularization is about
approximating that limit.

\begin{remark} \label{rem:gibbsplaneopt}
    Theorem \ref{thm:opt} implies that adding the condition $X \in {\rm GV}({\cal L}_\epsilon)$ to \eqref{eq:regSDP} leaves the optimizer unchanged. Hence, if we know equations for the Gibbs variety, we may shrink the feasible region by adding polynomial constraints. Most practical are the affine-linear equations: imposing $X \in {\rm GP}({\cal L}_\epsilon)$ allows to solve \eqref{eq:regSDP} on a spectrahedron of lower dimension.
\end{remark}

To prove Theorem \ref{thm:opt}, we derive two key properties of  the von Neumann entropy:

\begin{proposition} \label{prop:vonneumann}
The function $h$ satisfies:   
\begin{itemize}
\item[(a)] $h$ is strictly concave on the PSD cone $\mathbb{S}^n_+$, and \vspace{-0.1in}
\item[(b)] the gradient of $h$ is the negative matrix  logarithm: $\,\nabla(h)(X) = -{\rm log}(X)$.
\end{itemize}
\end{proposition}

\begin{proof}
For (a), we use a classical result by Davis \cite{Dav}. The function $h$ is invariant
in the sense that its value $h(X)$ depends on the eigenvalues of $X$. In fact, it is
a symmetric function of the $n$ eigenvalues $\lambda_1,\lambda_2,\ldots,\lambda_n$.
This function equals  $h(\lambda_1,\lambda_2,\ldots,\lambda_n) = \sum_{i=1}^n (\lambda_i-\lambda_i {\rm log}(\lambda_i))$,
and this is a concave function $\RR^n_+ \rightarrow \RR$. The assertion hence
follows from the theorem in \cite{Dav}.

For (b) we prove a more general result. For convenience, we change variables $Y = X - {\rm id}_n$ so that $f(Y) = h(Y + {\rm id}_n)$ is analytic at $Y = 0$. Fix any function
   $f: \RR \rightarrow \RR$  that is analytic in a neighborhood
of the origin. Then $Y \mapsto {\rm trace}(f(Y))$ is a well-defined real-valued
analytic function of $n \times n$ matrices $Y = (y_{ij})$ that are close to zero. 
The gradient of this function is the $n \times n$ matrix 
whose entries are the partial derivatives
$\partial {\rm trace}( f(Y))/\partial y_{ij}$. We claim that
\begin{equation}
\label{eq:matrixgradient} \nabla {\rm trace}(f(Y)) \, = \, f'( Y^\top) .  \end{equation}
Both sides are linear in $f$, and $f$ is analytic, so it suffices to prove this
for monomials, i.e.
\begin{equation}
\label{eq:matrixmonomials} \nabla {\rm trace}(Y^k) \, = \, k \cdot (Y^\top)^{k-1} \qquad \hbox{for all integers} \,\, k \geq 1. 
\end{equation}
Note that ${\rm trace}(Y^k)$ is a homogeneous polynomial of degree $k$ in the matrix entries $y_{ij}$,
namely it is the sum over all products $ y_{i_1 i_2} y_{i_2 i_3} \cdots y_{i_{k-2} i_{k-1}} y_{i_{k-1} i_1}$
that represent closed walks in the complete graph on $k$ nodes. When taking the derivative $\partial/\partial y_{ij}$ 
of that sum, we obtain $k$ times the sum
over all walks that start at node $j$ and end at node $i$. Here each  walk occurs with the factor $k$ 
because $y_{ij}$ can be inserted in $k$ different ways to create one of the closed walks above.
This polynomial of degree $k-1$ is the entry of the matrix power $Y^{k-1}$ in row $j$ and column $i$, so it is
the entry of its transpose $ (Y^\top)^{k-1}$ is row $i$ and column $j$. To prove
the proposition,
we now apply (\ref{eq:matrixgradient}) to the function
$f(y) =(y+1) - (y+1) \cdot {\rm log}(y+1)$.
\end{proof}

 If $\mathcal{L}=\mathcal{D}$ consists of diagonal matrices then the Gibbs manifold
${\rm GM}(\mathcal{D})$ is a discrete exponential family \cite[\S 6.2]{Seth}, and
$\pi({\rm GM}(\mathcal{D}))$ is the associated convex polytope.
This uses the  moment map from
toric geometry \cite[Theorem 8.24]{MS}. In particular, if  the linear space $\mathcal{D}$ is defined over $\mathbb{Q}$
then the polytope is rational and the Zariski closure of
${\rm GM}(\mathcal{D})$ is the toric variety of that polytope.
If the  space $\mathcal{D}$ is not defined over $\mathbb{Q}$
then ${\rm GM}(\mathcal{D})$ is an analytic toric manifold, whose Zariski closure is the
larger toric variety  ${\rm GV}(\mathcal{D}) = {\rm GM}(\mathcal{D}_\mathbb{Q})$ 
 seen in~(\ref{eq:toric}).

The key step to proving Theorem \ref{thm:opt} is a non-abelian version of the toric moment map.

\begin{theorem} \label{thm:gibbs}
The restriction of the linear map  $\pi : \mathbb{S}^n_+ \rightarrow \RR^d$
 to the Gibbs manifold ${\rm GM}(\mathcal{L})$
defines a bijection between $\,{\rm GM}(\mathcal{L})$ and the
open spectrahedral shadow $\,{\rm int}( \pi(\mathbb{S}^n_+))\,$ in $\,\RR^d$.
\end{theorem}

\begin{proof}
Fix an arbitrary positive definite matrix $X \in {\rm int}(\mathbb{S}^n_+)$ and set $b = \pi(X)$.
We must show that the spectrahedron $\pi^{-1}(b)$ contains precisely one
point that lies in ${\rm GM}(\mathcal{L})$.

Consider the restriction  of the von Neumann entropy $h$ to the spectrahedron $\pi^{-1}(b)$.
This restriction is strictly concave on the convex body $\pi^{-1}(b)$ by Proposition \ref{prop:vonneumann}.
Therefore  $h$ attains a unique maximum $X^*$ in the relative interior of $\pi^{-1}(b)$.
The first order condition at this maximum tells us that
$\nabla(h)(X^*) = -{\rm log}(X^*)$ lies in $\mathcal{L}$, which 
is the span of the gradients of the constraints $\langle A_i ,X  \rangle = b_i$.
Hence, the optimal matrix $X^*$ lies in the Gibbs~manifold 
$$ {\rm GM}(\mathcal{L}) \,\,=\,\,  \bigl\{ X \in \mathbb{S}^n_+ : {\rm log}(X) \in \mathcal{L} \bigr\}. $$
The assignment $b \mapsto X^*= X^*(b)$ is well defined and continuous on the interior of
the cone $\pi(\mathbb{S}^n_+)$. We have shown that it is a section of the
linear map $\pi$, which means 
$\pi( X^*(b)) = b$. We conclude that $\pi$ defines a homeomorphism between
${\rm GM}(\mathcal{L})$ and $\,{\rm int}( \pi(\mathbb{S}^n_+))$.
\end{proof}

\begin{proof}[Proof of Theorem \ref{thm:opt}]
    For any fixed $\epsilon > 0$, any minimizer $X^* = X^*_\epsilon$ of the 
    regularized problem \eqref{eq:regSDP}  in the spectrahedron
$\pi^{-1}(b)$ must satisfy
$\,C + \epsilon \cdot {\rm log}(X^*) \in \mathcal{L}$. 
This is equivalent to $X^* \in {\rm GM}(\mathcal{L}_\epsilon)$, and it
follows from the first order optimality conditions.
By the same convexity argument as in the proof of Theorem \ref{thm:gibbs},
the objective function in  \eqref{eq:regSDP} has only one
critical point in the spectrahedron $\pi^{-1}(b)$. This implies
$\,\pi^{-1}(b) \cap {\rm GM} (\mathcal{L}_\epsilon) \, = \,\{X^*_\epsilon\}$.
\end{proof}

We can now turn the discussion around and offer a definition of
Gibbs manifolds and Gibbs varieties purely in terms of 
convex optimization. Fix any LSSM $\mathcal{L}$ of dimension $d$ in $ \mathbb{S}^n$.
This defines a canonical linear map $\pi: \mathbb{S}^n_+ \rightarrow 
\mathbb{S}^n/\mathcal{L}^\perp \simeq \RR^d$.
Each fiber $\pi^{-1}(b)$ is a spectrahedron. If this is non-empty then
the entropy $h(X)$  has a unique maximizer 
$X^*(b)$ in  $\pi^{-1}(b)$. The Gibbs manifold
${\rm GM}(\mathcal{L})$ is the set of these
entropy maximizers $X^*(b)$ for $b \in \RR^d$.
The Gibbs variety ${\rm GV}(\mathcal{L})$ is
defined by all polynomial constraints satisfied by these $X^*(b)$.

This extends naturally to any ASSM $A_0 + \mathcal{L}$.
We now maximize the concave function
$h(X) + \langle A_0, X\rangle $ over
the spectrahedra $\pi^{-1}(b)$.
The Gibbs manifold ${\rm GM}(A_0+\mathcal{L})$
collects all maximizers,
and the Gibbs variety ${\rm GV}(A_0 + \mathcal{L})$ is
defined by their polynomial constraints.

\begin{example}
Let $\mathcal{L}$ denote the space of
all Hankel matrices $[y_{i+j-1}]_{1 \leq i,j \leq n}$ in $\mathbb{S}^n$.
This LSSM has dimension $d = 2n-1$.
The linear map $\pi : \mathbb{S}^n_+ \rightarrow  \RR^d$
takes any positive definite matrix $X$
to a nonnegative polynomial $b = b(t)$ in one variable $t$ of degree $2n-2$.
We have $ b(t) = (1,t,\ldots,t^{n-1})  X (1,t,\ldots,t^{n-1})^\top $,
so the matrix $X$ gives a sum-of-squares  (SOS) representation of $b(t)$.
The fiber $\pi^{-1}(b)$ is the {\em Gram spectrahedron}
\cite{Sch}
of the polynomial $b$.  The entropy maximizer
$X^*(b)$ in the Gram spectrahedron
is a favorite SOS representation
of~$b$. The Gibbs manifold ${\rm GM}(\mathcal{L})$
gathers the favorite SOS representations for all
non-negative polynomials $b$.
The Gibbs variety ${\rm GV}(\mathcal{L})$, which has dimension
$\leq 3n-2$, is the tightest outer approximation
of ${\rm GM}(\mathcal{L})$ that is definable by polynomials
in the matrix entries.

In Example \ref{ex:gram} we saw a variant of
$\mathcal{L}$, namely the sub-LSSM 
 where the upper left entry
 of the Hankel matrix was fixed to be zero.
If $C =-E_{11}$ is the corresponding negated matrix unit,
then (\ref{eq:SDP})
is the problem of minimizing $b(t)$ over $t \in \RR$.
See \cite[Section 12.3]{MS} for a first introduction to
polynomial optimization via SOS representations.
It would  be interesting to explore the potential of the
entropic regularization  (\ref{eq:regSDP})  for 
 polynomial optimization.
 \end{example}

One of the topics of \cite{ourteam}
was a scaling algorithm for solving the optimization
problem (\ref{eq:regSDP}) for linear programming (LP),
i.e.~the case when $A_1,\ldots,A_d$ are diagonal matrices.
This algorithm extends the Darroch-Ratcliff algorithm
for Iterative Proportional Fitting in statistics.
Combining this with 
a method for driving $\epsilon $ to zero leads to 
a numerical algorithm
for large-scale LP problems,
such as the optimal transport problems
in \cite[Section 3]{ourteam}.

We are hopeful that the scaling algorithm
can be extended to the problem (\ref{eq:regSDP}) in
full generality. 
By combining this with 
a method for driving $\epsilon $ to zero, 
one obtains a numerical framework
for solving SDP problems
such as quantum optimal transport
in Section \ref{sec:6}.

One important geometric object for SDP
 is the limiting Gibbs manifold, ${\rm lim}_{\epsilon \rightarrow 0} \,{\rm GM}(\mathcal{L}_\epsilon)$.
This is the set of  optimal solutions, as $b$ ranges over $\RR^d$.
In the case of LP, with $C$ generic, it is the simplicial complex which forms the
regular triangulation given by $C$.
This reveals the combinatorial essence of entropic regularization of LP,
  as explained in
\cite[Theorem~7]{ourteam}. 
From the perspective of {\em positive geometry}, it would
be worthwhile to study 
 ${\rm lim}_{\epsilon \rightarrow 0}\,{\rm GM}(\mathcal{L}_\epsilon)$
for~SDP. This set is semialgebraic, and it defines a nonlinear subdivision of the spectrahedral
shadow $\pi(\mathbb{S}^n_+)$.
If we vary the cost matrix $C$, the theory
of {\em fiber bodies} in \cite{MM} becomes relevant.

\section{Quantum optimal transport} \label{sec:6}

In this section we examine a semidefinite programming analogue
of the classical optimal transport problem,
known as {\em quantum optimal transport} (QOT). We follow the
presentation by Cole, Eckstein, Friedland, and Zyczkowski in  \cite{CEFZ}.
Our notation for the dimensions is as in \cite[Section 3.1]{ourteam}.
We consider the space $\mathbb{S}^{d_1 d_2}$ of
real symmetric matrices $X$ of size $d_1 d_2 \times d_1 d_2$. 
Rows and columns are indexed by $[d_1] \times [d_2]$.
Thus, we write $X = (x_{ijkl})$, where $(i,j)$ and $(k,l)$ are in $[d_1] \times [d_2]$.
The matrix being symmetric means that $x_{ijkl} = x_{klij}$ for all indices.
Each such matrix is mapped to a pair of two {\em partial traces} by the following linear~map:
$$  \mathbb{S}^{d_1 d_2} \,\rightarrow \, \mathbb{S}^{d_1} \times \mathbb{S}^{d_2}\, , \,\,\,
X \mapsto (Y,Z),  $$
where the $d_1 {\times} d_1$ matrix $Y = (y_{ik})$ satisfies $y_{ik} = \sum_{j=1}^{d_2} x_{ijkj}$,
and the $d_2 {\times} d_2$ matrix $Z = (z_{jl})$ satisfies $z_{jl} = \sum_{i=1}^{d_1} x_{ijil}$.
If $X$ is positive semidefinite then so are its partial traces $Y$ and~$Z$.
Hence our {\em marginalization map} restricts to a linear projection of closed convex cones, denoted
\begin{equation}
\label{eq:mumap}
  \mu \,:\, \mathbb{S}_+^{d_1 d_2} \,\rightarrow \, \mathbb{S}_+^{d_1} \times \mathbb{S}_+^{d_2}\, , \,\,\,
X \mapsto (Y,Z).  
\end{equation}
Diagonal matrices in $\mathbb{S}_+^{d_1 d_2}$
can be identified with rectangular matrices of format $d_1 \times d_2$ whose entries are nonnegative.
The map $\mu$ takes such a rectangular matrix to its
row sums and column sums. 
Hence the restriction of $\mu$ to diagonal matrices in  $\mathbb{S}_+^{d_1 d_2}$ is precisely the 
linear map that defines classical optimal transport in the discrete setting of \cite[Section 3.1]{ourteam}.

The quantum optimal transportation problem (QOT) is the task of minimizing
a linear function $X \mapsto \langle C, X \rangle $ over any
  transportation spectrahedron $\mu^{-1}(Y,Z)$. This is an SDP. 
  Our main theorem in this section states that its  Gibbs manifold is semialgebraic.

\begin{theorem} \label{thm:QOT}
The Gibbs manifold ${\rm GM}(\mathcal{L})$ for QOT is a semialgebraic subset
of $ \mathbb{S}^{d_1d_2}_+$. It  consists of all symmetric
matrices $Y \otimes Z$, where $Y \in \mathbb{S}^{d_1}_+$ and $Z \in \mathbb{S}^{d_2}_+$.
The Gibbs variety  ${\rm GV}({\cal L}) \subset \mathbb{S}^{d_1d_2}$ is linearly isomorphic to the cone over the
Segre variety $\,\PP^{\binom{d_1+1}{2}-1} \times \PP^{\binom{d_2+1}{2}-1} $.
\end{theorem}

The image of the marginalization map $\mu$ generalizes the polytope
$\Delta_{d_1-1} \times \Delta_{d_2-1}$, and the fibers of $\mu$ are quantum versions of
transportation polytopes.
These shapes are now~nonlinear.

\begin{lemma} \label{lem:onfire}
The image of the map $\mu$ is a convex
cone of dimension  $\binom{d_1+1}{2} + \binom{d_2+1}{2} - 1$:
\begin{equation}
\label{eq:image} {\rm image}(\mu) \,\,\, = \,\,\,
\bigl\{ (Y,Z) \in \mathbb{S}_+^{d_1} \times \mathbb{S}_+^{d_2} \,: \, {\rm trace}(Y) = {\rm trace}(Z) \bigr\} .
\end{equation}
For any point $(Y,Z)$ in the relative interior of this cone, the {\em transportation spectrahedron}
$\mu^{-1}(Y,Z)$ is a compact convex body of dimension
$\frac{1}{2} (d_1-1)(d_2-1) (d_1 d_2 + d_1 + d_2 + 2)$.
\end{lemma}

\begin{proof}[Proof of  Lemma \ref{lem:onfire}]
The partial trace map $\mu$  in (\ref{eq:mumap}) 
restricts to tensor products  as follows:
\begin{equation}
\label{eq:restrictedmu} \mu(Y \otimes Z) \, = \, \bigl(\, {\rm trace}(Z) \cdot Y \,,\, {\rm trace}(Y) \cdot Z \,\bigr). 
\end{equation}
Hence, if  $Y \in \mathbb{S}^{d_1}_+$ and 
$Z \in \mathbb{S}^{d_2}_+$ satisfy $t = {\rm trace}(Y) = {\rm trace}(Z)$
then $\frac{1}{t} Y \otimes Z$ is a positive semidefinite matrix
in the fiber $\mu^{-1}(Y,Z)$. This shows that the image  is as
claimed on the right hand side of (\ref{eq:image}). The image
is a spectrahedral cone of dimension $\binom{d_1+1}{2} {+} \binom{d_2+1}{2} {-} 1$.
Subtracting this from ${\rm dim}\,\mathbb{S}^{d_1 d_2}_+ = \binom{d_1 d_2 + 1}{2}$
yields the dimension of the interior fibers.
\end{proof}

\begin{example}[$d_1{=}d_2{=}2$] \label{ex:QOT}
The map $\mu$ projects positive semidefinite $4 \times 4$ symmetric matrices
$$ X \,=\, \begin{small} \begin{bmatrix}
 x_{1111} & x_{1112} &	x_{1121} &  x_{1122} \\
 x_{1112} & x_{1212} &	x_{1221} & x_{1222} \\
 x_{1121} & x_{1221} &	x_{2121} & x_{2122} \\
 x_{1122} & x_{1222} &	x_{2122} & x_{2222}
\end{bmatrix} \end{small}
$$ 
 onto a $5$-dimensional convex cone, given by the direct product of two disks. The formula is
 $$ Y \,=\,\begin{bmatrix}
   x_{1111}+x_{1212} & 
 x_{1121}+x_{1222} \\
 x_{1121}+x_{1222} & x_{2121}+x_{2222} \end{bmatrix} \quad {\rm and} \quad
Z \,= \, \begin{bmatrix}
  x_{1111}+x_{2121} &
  x_{1112}+x_{2122} \\
   x_{1112}+x_{2122} &
 x_{1212}+x_{2222} \end{bmatrix}. $$
The fibers of this map $\mu$ are the $5$-dimensional transportation spectrahedra $\mu^{-1}(Y,Z)$.

To illustrate the QOT problem, we fix the margins  and the cost matrix as follows:
\begin{equation}
\label{eq:YZinstance} Y \, = \, \begin{bmatrix} 5 & 1 \\ 1 & 6 \end{bmatrix}
\quad {\rm and} \quad
     Z \, = \, \begin{bmatrix} 7 & 2 \\ 2 & 4 \end{bmatrix} \quad {\rm and} \quad
     C \,\, = \,\, \begin{small} \begin{bmatrix}
     2 & 3 & 5 & 7 \\
     3 & 11 & 13 & 17 \\
     5 & 13 & 23 & 29 \\
     7 & 17 & 29 &31 \end{bmatrix} \end{small}.
\end{equation}     
We wish to minimize $\langle C, X \rangle $ subject to $\mu(X) = (Y,Z)$.
The optimal solution $X^*$ is equal to
$$ \!\!\! \begin{tiny} \begin{bmatrix}
    3.579128995196972555885181314 & 2.148103387337332721011731020 &
       2.671254991031789281229265149 &  -2.07566204542024789990696017 \\
    2.148103387337332721011731020 & 1.420871004803027444114818686 &
      1.16978382139276763200240537 & \!\!\! -1.671254991031789281229265149 \\
    2.671254991031789281229265149 & 1.169783821392767632002405371 &
      3.420871004803027444114818686 & -0.14810338733733272101173102\\
    -2.07566204542024789990696017 & -1.671254991031789281229265149 &
      -0.14810338733733272101173102 & 2.579128995196972555885181314 \\
      \end{bmatrix}\! .\end{tiny}
      $$
This matrix has rank $2$.
The optimal value equals
$v = 156.964485798827271035367539305...$. This is an algebraic number of degree $12$.
Its exact representation is the minimal polynomial 
$$ \begin{small} \begin{matrix} \! 125 v^{12}-465480 v^{11}+770321646 v^{10}{-}744236670798 v^9
{+}463560077206539
      v^8{-}193865445786866004v^7 \\ +54901023652716544539v^6-
      10330064181552258647604 v^5+1219620644420527588643307 v^4 \\ -
      77994100149206862070472310 v^3+1395374211380010273312826701 v^2 \\ +
      83502957914204004050312708316 v-2047417613706778627978564647804 \,\,=\,\,0.
      \end{matrix} \end{small}
$$
This was derived from the
KKT equations in \cite[Theorem 3]{NRS}. We conclude that the algebraic degree of
QOT for $d_1 = d_2 = 2$ is equal to $12$. This is smaller than
the algebraic degree of semidefinite programming, which is $42$. That
is the entry for $m{=}5$ and $n{=}4$ in \cite[Table~2]{NRS}.

This drop arises because QOT is a very special SDP.
The LSSM for our QOT problem~is
\begin{equation} \label{eq:LSSM4} \mathcal{L} \,\, = \,\,
\left\{ \begin{small}
\begin{bmatrix}
y_1 + y_3 & y_5 & y_4 &  0 \\
y_5 &  y_1 & 0 & y_4 \\
 y_4 & 0 & y_2+y_3 & y_5 \,\\
 0 & y_4 &  y_5 & y_2 \,
 \end{bmatrix} \end{small} \,:\,\,
 y_1,y_2,y_3,y_4,y_5 \in \RR\, \right\}. 
 \end{equation}
 This defines our $5$-dimensional Gibbs manifold ${\rm GM}(\mathcal{L})$ in
the $10$-dimensional cone $\mathbb{S}^4_+$. 
Theorem~\ref{thm:QOT} states that it equals the positive part of the Gibbs variety,
i.e.~${\rm GM}(\mathcal{L}) = {\rm GV}(\mathcal{L}) \cap \mathbb{S}^4_+$.

We compute the entropy maximizer
inside the $5$-dimensional transportation spectrahedron $\mu^{-1}(Y,Z)$
for the  marginal matrices $Y$ and $Z$ in (\ref{eq:YZinstance}).
Notably, its entries are rational:
\[
\mu^{-1}(Y,Z) \,\cap \, {\rm GV}(\mathcal{L}) \quad = \quad
 \mu^{-1}(Y,Z) \,\cap \, {\rm GM}(\mathcal{L}) \quad = \quad
\left\{ \,\frac{1}{11} \begin{small} \begin{bmatrix}
 35 &  10 &   7 &   2 \\
 10 & 20 &   2 &   4 \\
  7 &  2 &   42 &  12 \\
  2 &  4 &   12 &  24 \end{bmatrix}\end{small}\, \right\}. 
\qedhere \]
\end{example}

\begin{proof}[Proof of Theorem \ref{thm:QOT}]
By linear extension, the equation (\ref{eq:restrictedmu}) serves as a definition of 
the marginalization map $\mu$ on $\mathbb{S}^{d_1 d_2}$. We observe the following for the
trace inner product on~$\mathbb{S}^{d_1 d_2}$:
$$ \begin{matrix} & {\rm trace} \bigl( (A \otimes {\rm id}_{d_2}) (Y \otimes Z) \bigr) \,=
\, {\rm trace}(Z) \cdot {\rm trace}(AY)  & \hbox{for all}\,\,\, A \in \mathbb{S}^{d_1} \\
{\rm and} & {\rm trace} \bigl( ({\rm id}_{d_1} \otimes B) (Y \otimes Z) \bigr) \,=
\, {\rm trace}(Y) \cdot {\rm trace}(BZ)  & \hbox{for all}\,\,\, B \in \mathbb{S}^{d_2}. \\
\end{matrix}
$$
Therefore, the $(i,j)$ entry of ${\rm trace}(Z) \cdot Y$ is obtained as $\frac{1}{2}\langle (E_{ij} + E_{ji}) \otimes {\rm id}_{d_2}, Y \otimes Z \rangle$, where $E_{ij}$ is the $(i,j)$-th matrix unit. A similar observation holds for the entries of ${\rm trace}(Y) \cdot Z$. 
This means that $\mu(X)$ is computed by evaluating
$ {\rm trace} \bigl( (A \otimes {\rm id}_{d_2}) X \bigr)$
and ${\rm trace} \bigl( ({\rm id}_{d_1} \otimes B)X \bigr)$,
where $A$ ranges over a basis of $\mathbb{S}^{d_1}$
and $B$ ranges over a basis of $\mathbb{S}^{d_2}$. Therefore, we have
\begin{equation}
\label{eq:LSSMQOT}
 \mathcal{L} \,\, = \,\,
\bigl\{ \,A \otimes {\rm id}_{d_2}\, + \, {\rm id}_{d_1} \otimes B\,\,:\,
A \in \mathbb{S}^{d_1}\,\,{\rm and}\,\,
B \in \mathbb{S}^{d_2} \,\bigr\}. 
\end{equation}
Now, the key step in the proof consists of the following formula for the matrix logarithm
$$ {\rm log}(Y \otimes Z) \,= \,
{\rm log}(Y) \otimes {\rm id}_{d_2} \, + \,  {\rm id}_{d_1} \otimes {\rm log}(Z). 
$$
This holds for positive semidefinite matrices $Y$ and $Z$, and it
is verified by diagonalizing these matrices. By setting
$Y = {\rm exp}(A)$ and $Z = {\rm exp}(B)$, we now conclude that
the Gibbs manifold
${\rm GM}(\mathcal{L}) $ consists of all tensor products $Y \otimes Z$
where $Y \in \mathbb{S}^{d_1}_+$ and 
$Z \in \mathbb{S}^{d_2}_+$.

We have shown that ${\rm GM}(\mathcal{L})$ is
the intersection of
a variety with $\mathbb{S}^{d_1d_2}_+$. 
This variety must be the Gibbs variety ${\rm GV}(\mathcal{L})$. 
More precisely,
${\rm GV}(\mathcal{L}) $ consists of all tensor products $Y \otimes Z$
where $Y,Z$ are complex symmetric.
This is the cone over the Segre variety, which is the
projective variety in  $ \PP^{\binom{d_1d_2+1}{2}-1}$
whose points are the tensor products $ {X = Y \otimes Z}$.
\end{proof}

We have the following immediate consequence of the proof of Theorem \ref{thm:QOT}.
The entropy maximizers have rational entries.
This explains the matrix at the end of Example \ref{ex:QOT}

\begin{corollary} The Gibbs point for QOT is given by
$\frac{Y \otimes Z}{{\rm trace}(Y)}  $, with $Y, Z$ the given margins.
\end{corollary}

At this point, it pays off to revisit Section \ref{sec:3} and to study its thread
for the LSSM in~(\ref{eq:LSSMQOT}).

\begin{example}
We apply Algorithm \ref{alg:impl} to the LSSM $\mathcal{L}$ in (\ref{eq:LSSM4}).
The eigenvalues of $\mathcal{L}$ are distinct, and the ideal $\langle E_1' \rangle$ in step \ref{step:vieta}
is the intersection of six prime ideals. One of them~is
$$ \begin{matrix}
\langle\,
\lambda_1+\lambda_2-y_1-y_2-y_3 \, ,\,\,
\lambda_3+\lambda_4-y_1-y_2-y_3 \,, \qquad \qquad \qquad \\
2 \lambda_2 \lambda_4-\lambda_2 y_1-\lambda_4 y_1-\lambda_2 y_2-\lambda_4 y_2+2 y_1 y_2-\lambda_2 y_3-\lambda_4 y_3+y_1 y_3+y_2 y_3+y_3^2-2 y_4^2+2 y_5^2,\qquad \\ \qquad
 \lambda_2^2+\lambda_4^2-\lambda_2 y_1-\lambda_4 y_1-\lambda_2 y_2-\lambda_4 y_2+2 y_1 y_2-\lambda_2 y_3-\lambda_4 y_3+y_1 y_3+y_2 y_3-2 y_4^2-2 y_5^2\,
\rangle.
\end{matrix}
$$
The other five associated primes are found by permuting indices of $\lambda_1,\lambda_2,\lambda_3,\lambda_4$.
Hence, the Galois group $G_\mathcal{L}$ is the Klein four-group
$S_2 \times S_2$ in $S_4$, and we infer the linear relation $\lambda_1+\lambda_2-
\lambda_3-\lambda_4$.
The set $E_3$ in step \ref{step:tor} is the singleton
$\{z_1 z_2 - z_3 z_4\}$.
The elimination in step \ref{step:J} reveals the prime ideal
in $\RR[X]$ that is shown for arbitrary $d_1,d_2$ in Corollary \ref{cor:final}.
\end{example}

Our final result is derived from Theorem \ref{thm:QOT}
using tools of toric algebra \cite[Chapter 8]{MS}.

\begin{corollary} \label{cor:final}
The Gibbs variety for QOT is parametrized by  monomials $x_{ijkl} = y_{ik} z_{jl}$ that are not all distinct.
Its prime ideal in $\RR[X]$ is minimally generated by
the $2 \times 2$ minors of a 
matrix of format $\binom{d_1+1}{2} \times \binom{d_2+1}{2}$,
together with $\binom{d_1}{2} \binom{d_2}{2}$
linear forms in the entries of $X$.
\end{corollary}

We propose to extend QOT to
quantum graphical models \cite{WG}. In statistics, every undirected graph $G$
on $s$ vertices defines such a model \cite[Section 13.2]{Seth}.
The graphical model lives in the  probability 
simplex $\Delta_{d_1 d_2 \cdots d_s-1}$. Its points 
are nonnegative tensors of format 
$d_1 \times d_2  \times \cdots \times d_s$ whose
entries sum to $1$.
The quantum graphical model lives in the high-dimensional
PSD cone $\mathbb{S}^{d_1 d_2 \cdots d_s}_+$,
where the marginalization records the partial trace
for every clique in $G$. 
It would be interesting to
study the Gibbs manifold and the
Gibbs varieties for these models.
One may ask whether they agree for all graphs $G$ that are decomposable.
By Theorem \ref{thm:QOT},
this holds for QOT, where
$G$ is the graph with  two nodes and no edges.

\section*{Acknowledgements}
We are grateful to François-Xavier Vialard  and Max von Renesse for
inspiring discussions about Gibbs manifolds in optimization and optimal transport,
and we thank Benjamin Bakker for 
answering our questions about
transcendental number theory.
Simon Telen was supported by a Veni grant from the Netherlands Organisation for Scientific Research (NWO).

\bigskip
\bigskip

\noindent
\footnotesize
{\bf Authors' addresses:}

\smallskip

\noindent Dmitrii Pavlov,
MPI-MiS Leipzig
\hfill \url{dmitrii.pavlov@mis.mpg.de}

\noindent Bernd Sturmfels,
MPI-MiS Leipzig  and UC Berkeley
\hfill \url{bernd@mis.mpg.de}

\noindent Simon Telen,
CWI Amsterdam and MPI-MiS Leipzig
\hfill \url{Simon.Telen@cwi.nl}


\begin{thebibliography}{10}
\begin{small}
\setlength{\itemsep}{-0.6mm}

\bibitem{Ax} 
J.~Ax: {\em On Schanuel's conjectures}, Annals of Mathematics~{\bf 93} (1971) 252--268.

\bibitem{Briand} E.~Briand: {\em When is the algebra of multisymmetric polynomials generated by the 
elementary multisymmetric polynomials?}, Beitr\"age zur Algebra und Geometrie {\bf 45} (2004) 353--368.

\bibitem{CEFZ} S.~Cole, M.~Eckstein, S.~Friedland and K.~Zyczkowski:
{\em  Quantum Monge-Kantorovich problem and transport distance between density matrices},
Phys. Rev. Lett. {\bf 129} (2022) 110402.


\bibitem{Dav}
C.~Davis: {\em All convex invariant functions of hermitian matrices},
Arch.~Math.~{\bf 8} (1957) 276--278.

\bibitem{DS} P.~Diaconis and B.~Sturmfels:
{\em Algebraic algorithms for sampling from conditional distributions},
 Annals of Statistics {\bf 26} (1998) 363--397.

\bibitem{FMS} 
C.~Fevola, Y.~Mandelshtam and B.~Sturmfels:  {\em Pencils of quadrics: old and new}, 
Le Matematiche~{\bf 76} (2021) 319--335.

\bibitem{FDW} J.~Forsg\aa rd and T.~de Wolff:
{\em The algebraic boundary of the sonc-cone},
{\em SIAM J.~Appl.~Algebra Geom.} {\bf 6} (2022) 468--502. 

\bibitem{GS}
F.~Galuppi and M.~Stanojkovski:
{\em Toric varieties from cyclic matrix semigroups}, Rend. Istit. Mat. Univ. Trieste {\bf 53} (2021), Art.~No.~17.

\bibitem{Gir1} K.~Girstmair: {\em Linear dependence of zeros of polynomials and construction of primitive elements},
Manuscripta Math.~{\bf 39} (1982) 81--97. 

\bibitem{Gir2}
K.~Girstmair: {\em Linear relations between roots of polynomials},
 Acta Arith.~{\bf 89} (1999) 53--96.
 
\bibitem{HRS} 
J.~Hauenstein, J.~Rodriguez and F.~Sottile:
{\em Numerical computation of Galois groups}, Found.~Comput.~Math. {\bf 18} (2018)  867--890. 


\bibitem{HJ}
R.~Horn and C.~Johnson: {\em Topics in Matrix Analysis}, Cambridge University Press, 1991.
  
 \bibitem{Yuhan}  Y.~Jiang and B.~Sturmfels:
 {\em Bad projections of the PSD cone},
 Collectanea Mathematica {\bf 72} (2021) 261--280.


\bibitem{Kit}
Y.~Kitaoka: {\em Notes on the distribution of roots modulo a prime of a polynomial},
Uniform Distribution Theory {\bf 12} (2017) 91-117.

\bibitem{MM} 
L.~Mathis and C.~Meroni: {\em Fiber convex bodies},
 Discrete Comput. Geom. (2022), online first.
 
\bibitem{MS} M.~Micha{\l}ek and B.~Sturmfels: {\em Invitation to Nonlinear Algebra},
Graduate Studies in Mathematics, vol 211, American Mathematical~Society, Providence, 2021.

\bibitem{NRS}
J.~Nie, K.~Ranestad and B.~Sturmfels:
{\em The algebraic degree of semidefinite programming},
  Mathematical Programming {\bf 122} (2010) 379--405. 

\bibitem{Oscar}
The OSCAR Team: {\em OSCAR -- Open Source Computer Algebra Research system, Version 0.11.0}, \url{https://oscar.computeralgebra.de}, 2022.

\bibitem{ASCB} L.~Pachter and B.~Sturmfels:
{\em Algebraic Statistics for Computational Biology},
Cambridge University Press, 2005. 
\bibitem{Sch} C.~Scheiderer: {\em Extreme points of Gram spectrahedra of binary forms}, 
Discrete Comput. Geom. {\bf 67} (2022) 1174--1190.

\bibitem{ourteam}
B.~Sturmfels, S.~Telen, F.-X.~Vialard and M.~von Renesse:
{\em Toric geometry of entropic regularization},
 presented at MEGA 2022, Krak\'ow,
{\tt arXiv:2202.01571}.

\bibitem{SU} B.~Sturmfels and C.~Uhler:
{\em Multivariate Gaussians, semidefinite matrix completion, and convex algebraic geometry},
 Annals of the Institute of Statistical Mathematics {\bf 62} (2010) 603--638. 

\bibitem{Seth} S.~Sullivant: {\em Algebraic Statistics},
 Graduate Studies in Mathematics, 194, American Mathematical Society, Providence, RI, 2018. 


\bibitem{Sylvester} 
J.~J.~Sylvester: {\em On the equation to the secular inequalities in the planetary theory}, Philosophical Magazine Series 5, 16:100 (1883) 267-269.

\bibitem{Vigo} E.~Vigoda: {\em Sampling from Gibbs distributions},
PhD Dissertation, Computer Science Dept., UC Berkeley, 1999.

\bibitem{WG} S.~Weis and J.~Gouveia: {\em
Quantum marginals, faces, and coatoms}, 
{\tt arXiv:2103.08360}.

\end{small}
\end{thebibliography}
\end{document}